\documentclass[11pt]{article}

\usepackage{geometry,amsmath,bbm,bm,amsthm}
\usepackage[utf8]{inputenc}
\usepackage[english]{babel}
\usepackage{parskip}
\usepackage{setspace}
\usepackage{xcolor}
\usepackage{dsfont}
\usepackage[numbers,sort&compress]{natbib}
\usepackage{mathrsfs} 
\usepackage{scalefnt}
\usepackage{booktabs}
\usepackage{multirow}
\usepackage{threeparttable}
\usepackage{url}

\usepackage{makecell}

\allowdisplaybreaks

\usepackage{graphicx}
\usepackage[space]{grffile}

\newtheoremstyle{sty} 
  {0.5em} 
  {\topsep} 
  {\itshape} 
  {} 
  {\bfseries} 
  {.} 
  {.5em} 
  {} 
  
\newtheoremstyle{sty2}
  {0.5em} 
  {\topsep} 
  {} 
  {} 
  {\bfseries} 
  {.} 
  {.5em} 
  {} 

\theoremstyle{sty}\newtheorem{theorem}{Theorem}[section]
\theoremstyle{sty2}\newtheorem{definition}{Definition}
\newtheorem{lemma}{Lemma}
\newtheorem*{theorem*}{Theorem}

\renewcommand {\l}{\left(}
\renewcommand {\r}{\right)}

\renewcommand {\v}{\text{Var}}
\newcommand {\E}{\mathbbm{E}}
\renewcommand {\P}{\mathbbm{P}}
\renewcommand {\d}{\text{d}}
\newcommand {\wh}{\widehat}

\newcommand{\er}{Erd\H{o}s-R\'enyi }

\newcommand{\indicator}[1]{\mathds{1}\{ #1 \}}

\title{Estimating the size of a hidden finite set: \\ large-sample behavior of estimators}

\author{Si Cheng$^1$, Daniel J. Eck$^2$, and Forrest W. Crawford$^{3,4,5,6}$ \\[1em]
1. Department of Biostatistics, University of Washington\\
2. Department of Statistics, University of Illinois Urbana-Champaign \\
3. Department of Biostatistics, Yale School of Public Health \\
4. Department of Statistics \& Data Science, Yale University \\
5. Department of Ecology \& Evolutionary Biology, Yale University \\
6. Yale School of Management }

\begin{document}

\maketitle

\begin{abstract}
A finite set is ``hidden'' if its elements are not directly enumerable or if its size cannot be ascertained via a deterministic query. In public health, epidemiology, demography, ecology and intelligence analysis, researchers have developed a wide variety of indirect statistical approaches, under different models for sampling and observation, for estimating the size of a hidden set. Some methods make use of random sampling with known or estimable sampling probabilities, and others make structural assumptions about relationships (e.g. ordering or network information) between the elements that comprise the hidden set.  
In this review, we describe models and methods for learning about the size of a hidden finite set, with special attention to asymptotic properties of estimators.  We study the properties of these methods under two asymptotic regimes, ``infill'' in which the number of fixed-size samples increases, but the population size remains constant, and ``outfill'' in which the sample size and population size grow together.  Statistical properties under these two regimes can be dramatically different. 
\\[1em]
\end{abstract}






\section{Introduction}

Estimating the size of a hidden finite set is an important problem in a variety of scientific fields.  Often practical constraints limit researchers' access to elements of the hidden set, and direct enumeration of elements may be impractical or impossible.  
In demographic, public health, and epidemiological research, researchers often seek to estimate the number of people within a given geographic region who are members of a stigmatized, criminalized, or otherwise hidden group \citep{bao2015estimating,handcock2014estimating,unaids2010guidelines,abdul2014estimating}.  For example, researchers have developed methods for estimating the number of 
homeless people \citep{killworth1998estimation,david2002estimating}, 
human trafficking victims \citep{shelton2015proposed,van2015estimating}, 
sex workers \citep{johnston2017measuring,khalid2014estimating,johnston2015estimating,karami2017estimating,vuylsteke2017estimating}, 
men who have sex with men \citep{ezoe2012population,paz2011how,khalid2014estimating,wang2015application,wesson2015if,quaye2015critique,johnston2015estimating,sabin2016availability,rich2017estimating}, 
transgender people \citep{mcfarland2017many,sabin2016availability},
drug users \citep{kaplan1993how,hickman2006estimating,kadushin2006scale,heimer2010estimation,salganik2011assessing,nikfarjam2016national,sabin2016availability,johnston2015estimating,hall2000many},  
and people affected by disease \citep{yip1995capture,wittes1968generalization,hook1995capture,robles1988application,karon2008estimating,brookmeyer1988method}.  
In ecology, the number of animals of a certain type within a geographic region is often of interest \citep{seber1973estimation,Corn1984extinction,Hadfield1993decimation,karanth1998estimation}. Effective wildlife protection, ecosystem preservation, and pest control require knowledge about the size of free-ranging animal populations \citep{schwarz1999review,Funk2003,joglar1996declining}.  
In intelligence analysis, military science, disaster response, and criminal justice applications, estimates of the size of hidden sets can give insight into the size of a threat or guide policy responses. Analysts may seek information about the number of 
combatants in a conflict,  
military vehicles \citep{ruggles1947empirical,goodman1952serial}, 
extremists \citep{davies2014framework},
terrorist plots \citep{kaplan2010terror,kaplan2012estimating}, 
war casualties \citep{sadosky2015blocking}, 
people affected by a disaster \citep{bernard2001estimating}, 
and the extent of counterfeiting \citep{wilson2016measuring}.

Statistical approaches to estimating the size of a hidden set fall into a few general categories.  Some approaches are based on traditional notions of random sampling from a finite population \citep{horvitz1952generalization,bickel1992nonparametric}.  Others leverage information about the ordering of units \citep{ruggles1947empirical,goodman1952serial}, or relational information about ``network'' links between units \citep{killworth1998estimation,zheng2006many,bernard2010counting,mccormick2010many,feehan2016generalizing,salganik2011assessing}. Single- or multi-step sampling procedures that involve record collection or ``marking'' of sampled units -- called capture-recapture experiments -- are common when random sampling is possible \citep{chapman1951some,darroch1958multiple,fienberg1972multiple,seber1973estimation,pollock1990statistical,hickman2006estimating}.  Sometimes exogenous, or population-level data can help: when the proportion of units in the hidden set with a particular attribute is known \emph{a priori}, then the proportion with that attribute in a random sample can be used to estimate the total size of the set \citep{zhang2007advantages,zhang2007estimating,kimber2008estimating,heimer2010estimation,quaye2015critique,safarnejad2017population}.  Still other methods use features of a dynamic process, such as the arrival times of events in a queueing process, to estimate the number of units in a hidden set \citep{kaplan2010terror,kaplan2012estimating}.  

Alongside these practical approaches, corresponding theoretical results provide justification for particular study designs and estimators, based on large-sample (asymptotic) arguments.  Guidance for prospective study planning often depends on asymptotic approximation. For example, sample size calculation may be based on asymptotic approximation if the finite-sample distribution of an estimator is not identified or hard to analyze \citep{cochran1977sampling, daniel1999biostatistics, lwanga1991sample}. In retrospective analysis of data and the comparison of statistical approaches, researchers may choose estimators based on large-sample properties like asymptotic unbiasedness, efficiency and consistency if closed-form expressions for finite-sample biases and variances are hard to derive \citep{witte1999asymptotic, eubank1992asymptotic}. 
Claims about the large-sample performance of estimators depend on specification of a suitable asymptotic regime, and it is well known that estimators can perform differently under different asymptotic regimes. Asymptotic theory in spatial statistics provides some perspective on what it means to obtain more data from the same source: informally, an ``infill'' asymptotic regime assumes a bounded spatial domain, with the distance between data points within this domain going to zero.  
An ``increasing domain'' or ``outfill'' asymptotic regime assumes that the minimum distance between any pair of points is bounded away from zero, while the size of the domain increases as the sample size increases.   The latter is usually the default asymptotic setting considered by researchers studying the properties of spatial smoothing estimators \citep{lahiri1996inconsistency,mardia1984maximum,cressie1993asymptotic}.  However, under infill asymptotics, these desirable asymptotic properties of smoothing estimators often do not hold: even when consistency is guaranteed, the rate of convergence may be different \citep{cressie2015statistics, stein2012interpolation, lahiri1996inconsistency, zhang2004inconsistent, chen2000infill}.  When the size of the population from which the sample is drawn is the estimand of interest, intuition about large-sample properties of estimators can break down, but a similar asymptotic perspective is useful in studying the properties of estimators for the size of a hidden set: an infill asymptotic regime takes the total population size to be fixed, while the number of samples from this population increases; the outfill regime permits the sample size and population size to grow to infinity together.  



In this paper, we review models and methods for estimating the size of a hidden finite set in a variety of practical settings.  
First we present a unified characterization of set size estimation problems, formalizing notions of size, sampling, relational structures, and observation.  We then introduce the non-asymptotic regime in which sample size tends to the population size, and define the ``infill'' and ``outfill'' asymptotic regimes in which the sample size and population size may increase.  
We investigate a range of problems, query models, and estimators, including the German tank problem, failure time models, the network scale-up estimator, the Horvitz-Thompson estimator, the multiplier method, and capture-recapture methods. We characterize consistency and rates of estimation errors for these estimators under different asymptotic regimes. 
We conclude with discussion of the role of substantive and theoretical considerations in guiding claims about statistical performance of estimators for the size of a hidden set.

\section{Setting and notation}

\subsection{Hidden sets}

Let $U$ be a set consisting of all elements from a specified target population. In general, $U$ can be discrete or continuous.  Let $\mu(\cdot)$ be a measure defined on $U$ such that $\mu(U)<\infty$.  The \emph{size} of $U$ is $\mu(U)$.  We call $U$ a \emph{hidden} set if the members of $U$ are not directly enumerable, or if its size $\mu(U)$ cannot be ascertained from a deterministic query.  When $U$ is a finite set of discrete elements, $\mu(U)=|U|:=N$ is the cardinality of $U$. 


We seek to learn about the size of $U$ by sampling its elements.  Define a probability space $(U,\mathcal F, \P)$, where $\mathcal F$ is a $\sigma$-field, and $\P$ is a probability measure on $(U,\mathcal F)$. 
The measure $\P$ represents a probabilistic query mechanism by which we may draw subsets of the elements of $U$.  For each possible sample $s\in\mathcal F$, defining $\P(s)$ gives a notion of \emph{random sampling}. Sequential sampling designs can be specified by defining the sequential sampling probabilities $\P(S_i=s_i|s_1,\ldots,s_{i-1})$. Sequential samples are denoted as $\bm s=(s_1, \ldots, s_k)$, and the sample size is defined as $|s_1|+\cdots+|s_k|$, the sum of the cardinality of each sample, which can be larger than $\mu(U)$ under with-replacement sampling. An estimator $\delta(\bm s)$ of $\mu(U)=N$ is a functional of $\mathcal F$ onto $\mathbbm R^+$ or $\mathbbm N$. 

Elements of the hidden set $U$, or of a sample $s$ from $U$, may have attributes, labels, or relational structures that permit estimation of $\mu(U)$ from a subset.  An element $i\in U$ may be labeled or have attributes $X_i$, which may be continuous, discrete, unordered, or ordered.  The elements of $U$ may be connected via a relational structure, such as a graph $G=(U,E)$, where the vertex set is $U$, and edges $\{i,j\}\in E$ represent relationships between elements.  Alternatively, the sampling mechanism may impose a structure on the elements of a sample: if $s_1\subseteq U$ and $s_2\subseteq U$ are samples from $U$, then the intersection $M=s_1 \cap s_2$ is the set of elements in both samples. An \emph{observation} on the sample $\bm s$ consists of statistics that reflect these attributes, labels or structures of the units in $\bm s$, such as the value of attributes $\{X_i\}$, network degrees in a graph or size of the intersection of samples $|M|$. 

An example serves to make this setting and notation more concrete.  Consider the problem of estimating the number of injection drug users in a city \citep[e.g.][]{kaplan1993how,heimer2010estimation,salganik2011assessing}.  This is an important task in public health research and drug use epidemiology because injection drug use may contribute to transmission of infectious diseases such as hepatitis C virus (HCV) and human immunodeficiency virus (HIV). Policymakers considering educational and intervention programs to mitigate the harms of injection drug use require accurate estimates of the size of the target population.  In this context, $U$ is the set of injection drug users in the city, and we wish to estimate the size of this set, $\mu(U)=|U|=N$. The probability space is $(U,\mathcal{F},\P)$, where $\mathcal{F}$ is a $\sigma$-field consisting of subsets of $U$, and $\P$ is a probabilistic query distribution assigning probabilities to each set in $\mathcal F$.  For example, if $s\in\mathcal F$ is a subset of $U$, then $\P(s)$ represents a mechanism for randomly sampling a subset of $|s|$ members of $U$.  An individual injection drug user $i\in U$ may have an attribute $X_i$ representing, for example, the number of times $i$ has experienced an overdose and been taken to the local hospital.  In addition, relational information may be available in the form of a graph or network $G=(U,E)$, where $E$ is the set of pairs $\{i,j\}$ that are ``connected'' via syringe sharing or social relationships.

\subsection{Asymptotic regimes}

We now formalize asymptotic regimes relevant for hidden set size estimation. 

\begin{definition}[Asymptotic regime]
Let $(U_t, \mathcal F_t, \P_t)$ be a probability space defined for each $t=1,2,\ldots$, and let  $\bm s_t=\{s_1^{(t)},\ldots, s_{k_t}^{(t)}\}$ be the set of $k_t$ samples from $U$, with $|\bm s_t|=\sum_{i=1}^{k_t} |s_i|$.  An asymptotic regime is a sequence $\{\bm s_t, U_t, \P_t\}_{t=1}^\infty$ such that the limits $\lim_{t\rightarrow\infty} |\bm s_t|$ and $\lim_{t\rightarrow\infty} \mu(U_t) $ exist (infinity included). 
\end{definition}

We first define the trivial finite-population regime, in which the sampled set approaches the fixed population $U$.
\begin{definition}[Finite-population regime] 
Let $U$ be a hidden discrete set of fixed size. The finite-population (non-asymptotic) regime is $U_t=U$ for all $t$ and $\bm s_t=U$ for all $t>t_0$, where $t_0<\infty$ is a positive integer. 
\label{def:finite}
\end{definition}

Next, we define the ``infill'' asymptotic regime that arises when sampling repeatedly (with replacement between different samples) from a set of fixed finite size. This regime is an example of a superpopulation model \citep{isaki1982survey,brewer1979class} which reproduces the original population $U_t=U$ for each $t$.  

\begin{definition}[Infill asymptotic regime] 
Let $(U_t=U, \mathcal F_t= \mathcal F, \P_t)$ be a sequence of probability spaces, where $\P_t$ assigns probability $\P(s_i^{(t)} | s_1^{(t)},\ldots, s_{i-1}^{(t)})$ to sequential samples $s_1^{(t)}, \ldots, s_{k_t}^{(t)}\in\mathcal F$ for any $t$. The infill asymptotic regime is a sequence $\{\bm s_t, U_t=U, \P_t\}_{t=1}^\infty$, where $|s_j^{(t)}|$ (any $j\in [k_t]$) and $\mu(U_t)$ are both fixed and bounded, and the number of samples $k_t\rightarrow\infty$ as $t\rightarrow\infty$. 
\label{def:infill}
\end{definition}

Sometimes it can be difficult to conceptualize sampling infinitely many times from $U$, or the sampling design may be subject to practical constraints, so that sampling only a single or fixed number of samples, or a fixed proportion of the total population, is allowed. 
It is therefore also reasonable to study the performance of estimators under an asymptotic regime in which a \emph{single} sample is obtained from the hidden set, where the size of the sample and hidden set may tend to infinity together.  

\begin{definition}[Outfill asymptotic regime]
Let $(U_t,\mathcal F_t,\P_t)$ be a sequence of probability spaces, where $\P_t$ assigns probability $\P(s_i^{(t)}|s_1^{(t)},\ldots,s_{i-1}^{(t)})$ to $s_1^{(t)},\ldots,s_{k_t}^{(t)}\in\mathcal F_t$ for any $t$. The outfill asymptotic regime is a sequence $\{\bm s_t,U_t, \P_t\}$ such that $\mu(U_t)\rightarrow\infty$ and $n_{i}^{(t)}:=|s_i^{(t)}|\rightarrow\infty$ with $n_{i}^{(t)}/\mu(U_t)\rightarrow c_i\in [0,\infty)$ for each $i\in[k_t]$ as $t\rightarrow\infty$, where $\lim_{t\rightarrow\infty}k_t$ may be finite or infinite. 
\label{def:outfill1}
\end{definition}

\begin{figure} 
  \centering 
  \includegraphics[width=12cm]{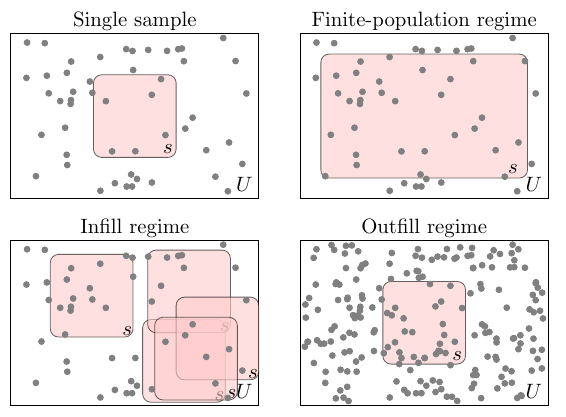}
  \caption{Illustration of different regimes for discrete sets. Units are indicated by circles. The sample $s$ ``expands'' to $U$ under the finite-population regime. Infinitely repeated samples of a fixed size are drawn from a fixed population under infill asymptotics. Under outfill, $s$ and $U$ grow simultaneously with $s$ approaching a fixed proportion of $U$.}
  \label{fig:asym}
\end{figure}

The ratio $c_i$ can be greater than one when sampling is with replacement. 
The sample sizes mentioned above can be deterministic or random. In the latter case, all regimes can be defined in a similar way, e.g. $\E |\bm s_t|/\mu(U_t)\rightarrow c_i$. 
We are primarily interested in the outfill asymptotic regime with $k_t=1$ for all $t$. The binomial model as well as the multiplier and capture-recapture methods, described below, are special cases where $k_t$ may be greater than one. Figure \ref{fig:asym} illustrates different regimes in general discrete settings.

\subsection{Statistical properties of estimators}

Let $\delta(\bm s_t)$ be an estimator of $\mu(U_t)$, defined for each $t$.  We are interested in the statistical properties of $\delta(\bm s_t)$ under the asymptotic regimes described above.  An estimator is called \textit{unbiased} if $\E_t[\delta(\bm s_t)]=\mu(U_t)$ for all $t$, where $\E_t(\cdot)$ denotes expectation with respect to $\P_t$.  Under an asymptotic regime $\{\bm s_t, U_t, \P_t\}_{t=1}^\infty$, an estimator $\delta(\bm s_t)$ is \textit{asymptotically unbiased} if $\lim_{t\rightarrow\infty}\E_t [\delta(\bm s_t)]-\mu(U_t)=0$.  There may be some slightly biased estimators whose variance is smaller than that of every unbiased estimator. A common way to balance the trade-off between the bias and variance is to evaluate the \textit{mean squared error} (MSE), defined as $MSE[\mu(U_t), \delta(\bm s_t)]=\E\left[(\delta(\bm s_t)-\mu(U_t))^2\right]=(\E[\delta(\bm s_t)]-\mu(U_t))^2+\v[\delta(\bm s_t)]$.  
The asymptotic MSE under a given regime is defined as $\lim_{t\rightarrow\infty} MSE(\mu(U_t),\delta(\bm s_t))$.

An estimator $\delta(\bm s_t)$ that satisfies  
$ \lim_{t\rightarrow\infty}\P_t(|\delta(\bm s_t) -\mu(U_t)|>\varepsilon)=0 $
for any $\varepsilon>0$ under a particular asymptotic regime $\{\bm s_t, U_t, \P_t\}$ is called
  \textit{consistent} for $\mu(U_t)$. 
An estimator $\delta(\bm s_t)$ is called \textit{MSE consistent} for $\mu(U_t)$ under a certain asymptotic regime if $MSE[\delta(\bm s_t), \mu(U_t)]\rightarrow 0$ as $t\rightarrow\infty$ under that asymptotic setting. MSE consistency implies consistency.
Under a particular asymptotic regime, we call a sequence of estimates $\delta(\bm s_t)$ \emph{asymptotically normal} with mean $\xi$, variance $\sigma^2/t^r$ and rate $t^r$ if the cumulative distribution function (CDF) of $t^r\l \delta(\bm s_t)-\xi \r$ converges to the CDF of a $N(0, \sigma^2)$ random variable, denoted by $t^r\l \delta(\bm s_t)-\xi \r \xrightarrow{L} N(0, \sigma^2)$.

\section{Ordered sets: the German tank problem}
\label{sec:ord}

Suppose each unit in the hidden set $i\in U$ has a distinct label $X_i\in \mathbbm R$, so that the labels give a natural ordering of the elements in $U$: we can define units $i<j$ if $X_i<X_j$. One common scenario for discrete $U$ is that the $X_i$'s are consecutive integers. Another common situation when $U$ is equivalent to an interval in $\mathbbm R$ is that $\cup _{i\in U}X_i$ equals that interval. An observation of samples from an ordered set $U$ consists of sampled units $s$ and their labels $\{x_i:\ i\in s\}$. 

In 1943, the Economic Warfare Division of the American Embassy in London initiated a project to learn about the capacity of the German military using serial numbers found on German equipment 
 \citep{ruggles1947empirical,GUM2005105}.
In a simple conceptualization of the problem, let $U=\{1,\ldots,N\}$ and consider sampling $n=|s|$ units without replacement from $U$ with probability $\P(s)=1/{N\choose n}$. 
With $k_t$ i.i.d. repeated samples, an estimator $\delta(\bm s)$ for $N$ is a functional of the observations, including the sample sizes and observed labels $X_{1,1}, \ldots, X_{1,n}, \ldots, X_{k_t,1},\ldots, X_{k_t, n}$. 
For example, to estimate the total number of participants in a marathon, if all $N$ participants are numbered by the consecutive integers $1,\ldots,N$, one could randomly record the first $n$ numbers they saw in the race, and estimate the total based on the observed numbers.

For the $k$th sample $X_{k,1}, \ldots, X_{k,n}$, we let $X_{k (n)}$ be the $n$th order statistic in the sample.  With one sample, the maximum likelihood estimator (MLE) for $N$ is $\wh N_{MLE}=X_{(n)}$, which is negatively biased. \citet{goodman1952serial} proposed an unbiased estimator
\begin{equation}
\wh N_G=\frac{n+1}{n}X_{(n)}-1,
\end{equation}
which is a uniformly minimum-variance unbiased estimator (UMVUE), with $\v(\wh N_G)=(N-n)(N+1)/n(n+2)$.  An alternative estimator of $N$ takes into account the gap between $X_{(n)}$ and $N$, and adjusts for the bias with the average gap between order statistics \citep{goodman1952serial}. The estimator
\begin{equation}
\wh N_2=X_{(n)}+\frac{X_{(n)}-X_{(1)}}{n-1}-1,
\end{equation}
is also unbiased, with $\v(\wh N_2)={n(N-n)(N+1)}/{(n-1)(n+1)(n+2)}$. The estimator $N_2$ can also be modified to estimate $N$ when the labels do not start with 1. In particular,
\begin{equation*}
\wh N_3=\frac{(n+1)\l X_{(n)}-X_{(1)} \r}{n-1}-1
\end{equation*}
is the UMVUE of $N$ when the initial label is unknown \citep{goodman1952serial}, with $\v (\wh N_3)=2(N-n)(N+1)/(n-1)(n+2)$.

\begin{figure} 
  \centering 
  \includegraphics[width=12cm]{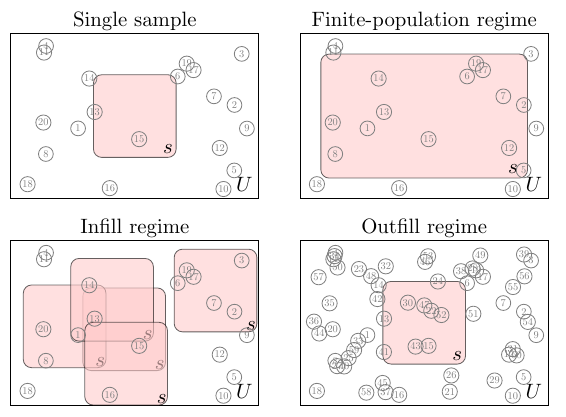}
  \caption{Illustration of a single sample, and the finite-population, infill, and outfill regimes for the German tank problem. Units with their labels are represented by circles with numbers inside.}
  \label{fig:german}
\end{figure}

When there is more than one sample, we take the MLE as the maximizer of the joint sampling probability $\P_t(s_1,\ldots, s_{k_t})$, which is $\max_{i\in[k_t]} X_{i(n)}$, the largest observed value across all $k_t$ samples. For estimators with closed forms like $\wh N_G, \wh N_2, \wh N_3$, we derive $k_t$ estimates $\delta(s_i^{(t)}), \ i=1, \ldots, k_t$ based on each sample, and take their average as the estimator. In remaining sections, we average the estimators under infill by default, except for the models where infinite without-replacement sampling is feasible (e.g. Section \ref{subsec:binom}). We consider the infill asymptotic regime where $n_t=n, N_t=N$ and $k_t\rightarrow\infty$, and the outfill regime where $n_t, N_t\rightarrow \infty, k_t=1$ with $n_t/N_t\rightarrow c\in (0,1)$. Figure \ref{fig:german} illustrates different regimes for the German tank problem. We have the following asymptotic results:

\begin{theorem}
Under the finite-population and infill regimes, $\wh N_{MLE}, \wh N_G, \wh N_2, \wh N_3$ are consistent.
  Under the outfill regime, all estimators above are asymptotically unbiased with asymptotic MSE $O(1)$ and inconsistent.
  Whether the initial label is known or not does not change the rate of MSE of the UMVUE.
  \label{thm_tank}
\end{theorem}


\section{Bernoulli Trials}
\label{sec:binom}

Consider a discrete hidden set $U$ consisting of $N$ unlabeled, indistinguishable units. A sample ${s}$ from $U$ arises by associating a binary indicator $Y_i \sim \text{Bernoulli}(p)$ to each $i\in U$, for fixed $0<p<1$, where different realizations of the $Y_i$'s can be generated in different draws. The probability $p$ may be known or unknown. A single sample consists of the subset of units with positive indicators, $s=\{i\in U:\ Y_i=1\}$. 
This is a frequently encountered situation in computer science, ecology, business, epidemiology, and many other fields \citep{friedman1999multicast, talluri2009finite, brookmeyer1988method, karon2008estimating}.

\subsection{Binomial $N$ parameter}
\label{subsec:binom}

We first assume that $p$ is known.  A single sample ${s}$ from $U$ gives a statistic $Q:=n=|s| = \sum_{i\in U} Y_i$ which has Binomial$(N,p)$ distribution. 
When there are $k$ independent samples, we assume they are generated by the same mechanism, so $\P(Q_1=q_1, \ldots, Q_k=q_k)=\prod_{i=1}^k{N\choose q_i}p^{q_i}(1-p)^{N-q_i}$. 
The method of moments estimator (MME) $\wh N_{MME}=\bar Q/p$ is an unbiased estimator of $N$. There are two versions of the MLE, derived from continuous and discrete likelihood equations respectively. The continuous MLE, $\wh N_{MLE}'$ is the solution of $\partial L/\partial N=0$ (take $Q_{(k)}$ if it is larger than the solution), and the discrete MLE $\wh N_{MLE}$ is the largest $N$ such that $L(N)-L(N-1)\ge 0$. 

\begin{figure} 
  \centering 
  \includegraphics[width=12cm]{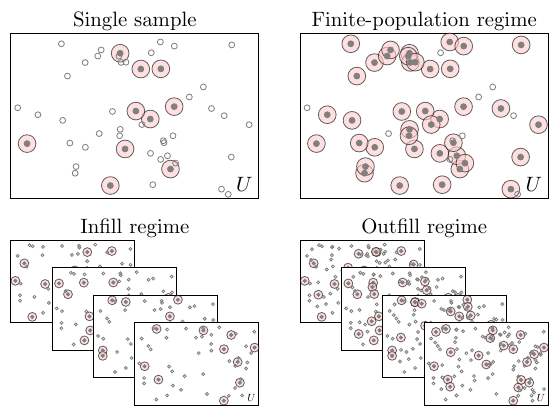}
  \caption{Illustration of the sampling mechanism for the binomial model, and the finite-population, infill and outfill asymptotic regimes. The solid points with red circles are units with indicator 1 (which are therefore in the sample), and the rest are unobserved.}
  \label{fig:bern}
\end{figure}

The finite-population regime arises when $k=1$ and $p\rightarrow 1$, i.e. when all units are associated with indicator $1$ and observed in a single sample. We consider the infill asymptotic regime with $N_t=N$ and $k_t\rightarrow\infty$.
The outfill regime is $k_t, N_t\rightarrow\infty$ with $k_t/N_t\rightarrow c>0$. Figure \ref{fig:bern} shows how the sampling mechanism varies under different regimes for the binomial $N$ model. 
The following theorem combines results in \citep{Blumenthal1980, Feldman1968} and states the consistency of estimates under the infill asymptotic regime, along with error rates under the outfill regime. In particular, the estimation error of $Q_{(k)}$ increases with $N$ under the outfill regime.

\begin{theorem}
Under the finite-population regime, $\wh N_{MME}$, $Q_{(k)}$ and $\wh N_{MLE}$ are consistent. 
Under infill asymptotics, $\wh N_{MLE}, \wh N_{MME}$, $Q_{(k)}$, and $\wh N'_{MLE}$ after rounding to the nearest integer, are consistent 
\citep{Blumenthal1980}. 
Under outfill asymptotics, $\wh N_{MME}$ and $\wh N'_{MLE}$ are both asymptotically unbiased and normal with variance $O(1)$. The ``relative error'' of the discrete MLE, ${(\wh N_{MLE}-N)}/{N^\alpha}\xrightarrow{P} 0$ for any $\alpha>1/2$. The ``relative error'' of $Q_{(k)}$ with $\alpha=1$ goes to $p-1$ in probability.
\label{thm:binom}
\end{theorem}

When $p$ is unknown, the situation does not improve: negative or unstable estimates may occur, and Bayesian approaches are usually adopted to avoid these issues. 
\citet{Blumenthal1980} adopted a conjugate prior Beta$(a,b)$ for $p$ and an improper uniform prior $p(N)\propto 1$ for $N$; the posterior is proper if and only if $a>1$ \citep{kahn1987cautionary}. 
\citet{Blumenthal1980} showed that 
the posterior mode $\wh N_m$ is consistent under infill asymptotics, and satisfies
\begin{equation*}
\frac{\sqrt n}{N}\l \wh N_m-N\r \xrightarrow{L} N\l 0, \frac{2(1-p)^2}{p^2} \r
\end{equation*}
under the outfill regime. In particular, the MSE rate is slower compared to $O(1)$ as in Theorem \ref{thm:binom} when $p$ is known.

A special case of the Binomial scenario arises for zero-truncated counts.  For example, a registry may record the number of times each unit has been observed, but zero counts are not recorded.  Distributional assumptions can be used to estimate the proportion of unobserved zero counts, leading to estimates of the set size.  Zero-truncated counting models have been used to estimate size of hard-to-reach populations, including drug users \citep{cruyff2008point, bohning2004estimating}, undocumented immigrants \citep{van2003point, bohning2009covariate}, criminal population \citep{van2003estimating, bouchard2007capture}, the number of infected households in an epidemic \citep{scollnik1997inference}, and species richness in ecology \citep{wilson1992capture, craig1953utilization}.  To illustrate, associate to each unit $i\in U$ a realization of the attribute $Y_i\sim\text{Poisson}(\lambda)$. 
A sample from $U$ is $s=\{i\in U: \ Y_i>0\}$ and an observation on $s$ is $\{Y_i:\ i\in s\}$, the set of all positive counts. 
For one sample, the sampling mechanism is given by 
$\P(y_1, \ldots, y_{|s|}|s)=\prod_{i\in s} \lambda^{y_i}/(e^{\lambda}-1)y_i!$. 
Estimating $\lambda$ under this model reveals the proportion of zero counts, $p=1-e^{-\lambda}$, and estimation of $N$ proceeds as in the Binomial$(N,p)$ case outlined above. 
The asymptotic results in Theorem \ref{thm:binom} follow. 

\subsection{Waiting times}

Sometimes the state of a hidden unit may change, thereby making it known to an observer. 
For example, terrorist plots may change state from ``hidden'' to ``executed'', making them observable by intelligence agents \citep{kaplan2010terror}. The temporal pattern of such state changes may give insight into the number of hidden units. Properties of waiting times to an event have been exploited to estimate the number of units in studies of terrorism, crime, and estimation of epidemiological risk population sizes \citep{kaplan2010terror,frey2010queue,crawford2016graphical,crawford2017hidden}.

Suppose $U$ is a set of $N$ hidden units in existence at time 0, each of which is at risk of ``failure'' at some future time.  To each $i\in U$, associate a failure time $T_i\sim\text{Exponential}(\lambda)$, and suppose failure times are observed up to some finite observation time $T>0$.  A sample is the set of units that have failed by the end of study, $s=\{i\in U:\ T_i<T\}$ with $|s|=n$, and an observation on $s$ is $\{T_i: i\in s\}$. With repeated sampling, a new observation is independent of all previous observations, taken after all units are set to be ``at risk'' over again. We consider the finite-population regime in which $T\rightarrow \infty$ so that all failures are observed, the infill regime in which $T$ and $N$ are fixed with the number of repeated observations $k_t\rightarrow\infty$, and the outfill regime in which $T_t, N_t\rightarrow\infty$ with $T_t/N_t\rightarrow c>0$. 
For example, if $U$ is the set of hidden terrorist plots \citep[e.g.][]{kaplan2010terror,kaplan2012estimating}, the finite-population regime keeps $|U|=N$ constant, while letting the maximum observation time $T\to\infty$, so that eventually every plot in $U$ is executed and thereby revealed to the observer. The infill regime consists of keeping $N$ and $T$ constant, while obtaining (hypothetical) repeated realizations of the same $N$ plots over $[0,T]$.  The outfill regime lets both the observation time $T$ and number of plots $N$ go to infinity together, so that more plots are added, while the observation time increases.  Figure \ref{fig:waiting} illustrates each regime under the waiting time model. 


\begin{figure}[!h] 
  \centering 
  \includegraphics[width=14cm]{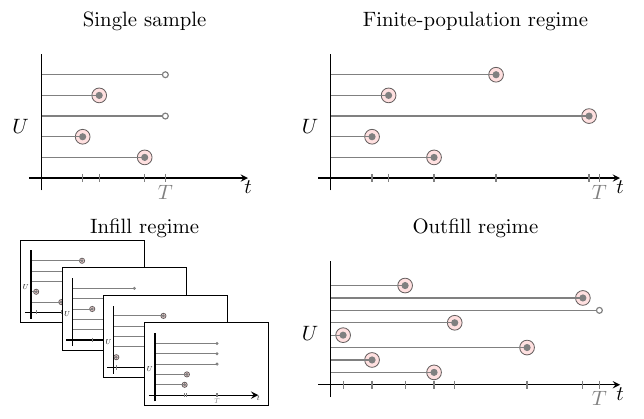}
  \caption{Illustration of the waiting time model. The observed event times are subject to right censoring at $t=T$, that is, events that occur before $T$ are observed. Solid dots with red shades indicate observed event times. The finite-population regime is that $T\rightarrow\infty$ so that all events are observed. Infill asymptotics amounts to generating different realizations of the failure times. Under the outfill regime, $T$ and the total number of units $N$ both increase toward infinity.}
  \label{fig:waiting}
\end{figure} 

 Let $\Delta_i:=T_i-T_{i-1}$ be the waiting time between the $(i-1)$th and $i$th failure. The sampling mechanism is given by 
\[ \P(t_1, \ldots, t_n | N,\lambda)=\lambda^n\cdot\prod_{i=1}^n (N-i+1)\cdot\exp\left[-\lambda\sum_{i=1}^n(N-i+1)\Delta_i\right]\cdot \exp[-\lambda(N-n)(T-t_n)], \]
which gives rise to the likelihood $L(t_1, \ldots, t_n; N)$. Alternatively, if we ignore the timing of events, the observed number of events can be characterized by a binomial model $\P(n | N, \lambda)={N\choose n}(1-e^{-\lambda T})^n e^{-\lambda T(N-n)}$, which yields $L_2(n; N)$. Maximizing $L$ and $L_2$ lead to two estimates, $\wh N_{MLE}$ and $\wh N'_{MLE}$ of $N$. It is easy to verify that $\partial \log L/\partial N=\partial \log L_2/\partial N$, so $\wh N_{MLE}$ and $\wh N'_{MLE}$ are identical. The timing of events does not contain more information about $N$ than the total number of events.  The asymptotic behavior of $\wh N_{MLE}$ follows from the discussion in Section \ref{subsec:binom}: when $\lambda$ is known, $\wh N_{MLE}$ is consistent under finite-population and infill regimes. Under the outfill regime, it is unbiased and asymptotically normal with variance $O(1)$.



\subsection{The network scale-up method}

Estimating the number of vertices in a hidden network or graph is an important problem in sociology, epidemiology, computer science, and intelligence applications \citep{killworth1998estimation, bernard2001estimating, zheng2006many, mccormick2010many, feehan2016generalizing,killworth1998social, katzir2011estimating}. A subgraph of a larger graph may contain information about the size of the larger graph \citep{feehan2016generalizing, bernard1991estimating, Massoulie2006peer}. The network scale-up method (NSUM) \citep{killworth1998estimation} provides an estimate for the size of a hidden population by making use of network information from a sub-sample of individuals. 

Consider a graph $G_V=(V,E)$, where $V$ is a set of $M$ units and $\{i,j\}\in E$ means that $i,j\in V$ are connected. $V$ is called the \textit{total population}, and a subset $U\subseteq V$ of size $N$ is the \textit{hidden population}. Assume $G_V$ is \emph{simple}, and has no parallel edges or self-loops.  The network of $U$ is $G_U=(U, E_U)$, where $E_U=\{\{i,j\}: i\in U, j\in U, \{i,j\}\in E\}$. We call $V\setminus U$ the \textit{general population}. A sample from a subset of $V$, along with network degrees of the sampled units within and outside of that subset provides information for learning about the size of $U$.  Suppose the total population network $G_V$ is generated from the \er random graph model \citep{erdos1959random} in which each pair of distinct vertices is connected independently by an edge with probability $\Pr(\{i,j\}\in E) = \pi$.  The likelihood of a random graph $G_V=(V,E)$ from the \er model with $|V|=N$ and connection probability $\pi$ is 
\[ \Pr(G_V) = \pi^{|E|} (1-\pi)^{\binom{N}{2}-|E|}  \]
where $|E|$ is the number of edges and $\binom{N}{2}$ is the number of unordered distinct pairs of vertices.  Two common sampling scenarios -- sampling from $V\setminus U$ and directly from $U$ -- are illustrated in Figure \ref{fig:nsum}.

\begin{figure} 
  \centering 
  \includegraphics[width=14cm]{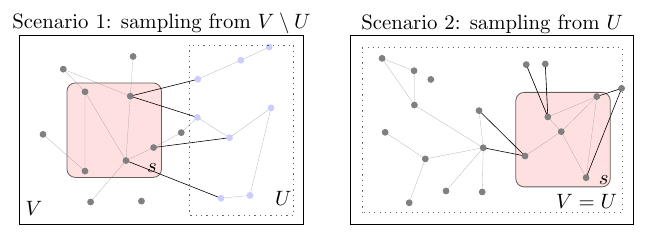}
  \caption{Illustration of the two common scenarios for the network scale-up method. In scenario 1, $U\subset V$ and we observe a randomly chosen subset $s$ of $V\setminus U$ and number of edges from each unit in $s$ to $U$ (thick lines) and to $V\setminus U$ (thin lines) respectively.  In scenario 2, $U=V$ and we observe the induced subgraph (edges represented by thin lines) from a randomly chosen subset $s$ of $U$ as well as the pendant edges (thick lines) between $s$ and $U\setminus s$.}
  \label{fig:nsum}
\end{figure} 


\subsubsection{Sampling from the general population}

We consider sampling uniformly at random from the general population $V\setminus U$ with a fixed sample size $|s|=n$. The sampling mechanism is $\P(s \mid |s|=n)=\binom{M-N}{n}^{-1}$. 
 For a sample $s\in V$, we observe network degrees $d_i^V:=\sum_{j\in V} \indicator{E_{ij}=1}$ and $d_i^{U}:=\sum_{j\in U} \indicator{E_{ij}=1}$ for each $i\in s$.
 As an empirical example, suppose we wish to estimate the number of people who died in an earthquake \citep[e.g.][]{bernard1991estimating}.  We cannot survey the dead (members of $U$) but we can survey living people ($V\setminus U$) to determine how many people they know ($d_i^V$), and how many they know who died as a result of the earthquake ($d_i^U$). 
 
Under the \er model, $\E d_i^V = (M-1)\pi\approx M\pi$ and $\E d_i^U = N\pi, \ \forall i\in V\setminus U$. Taking the ratio and canceling out $\pi$ yields the MME
\begin{equation}
\wh N_{NS}=M\cdot\frac{\sum_{i=1}^n d_i^U}{\sum_{i=1}^n d_i^V}.
\label{nsum2}
\end{equation}
Conditional on $d_i^V$, $d_i^U$ follows hypergeometric distribution for each $i$. The same estimator can also be derived under a different model assumption. \citet{killworth1998estimation} considered a model where $d_i^U$ is Binomial$(d_i^V, N/M)$ given $d_i^V$, and (\ref{nsum2}) is then the MLE under this binomial model, which is unbiased with variance $(M-N)N/\sum_i d_i^V$.

We consider the finite-population regime in which $n\rightarrow (M-N)$, i.e. $s\rightarrow V\setminus U$. Under the infill regime, $M, N, n$ are fixed and the number $k_t$ of repeated samples $s \subseteq V\setminus U$ goes to infinity. The outfill regime is that $M_t, N_t, n_t\rightarrow\infty$ such that $N_t/M_t\rightarrow c_1\in (0,1)$, $n_t/(M_t-N_t)\rightarrow c_2\in (0,1)$, and $k_t=1$.

Sometimes an intermediate step in deriving $\wh N_{NS}$ is the estimation of personal network sizes $d_i^V$. If unbiased estimates $\hat {d_i^V}$ are plugged in, $\wh N_{NS}$ would have a positive bias by Jensen's inequality since $1/x$ is a convex function. Let us assume for now that the $d_i^V$'s are observed true values. Theorem \ref{nsum_asym1} states the asymptotic properties of $\wh N_{NS}$ under the \er assumption. 

\begin{theorem}
$\wh N_{NS}$ has a positive bias $N/(M-1)$. It is not necessarily consistent under the finite-population regime, and converges to a positively biased quantity under infill. It is asymptotically normal with bias $c_1$ and variance $O(1)$ under the outfill regime.
\label{nsum_asym1}
\end{theorem}

\subsubsection{Sampling from the hidden population}

When possible, a random sample from the hidden population $U$ can also lead to a valid estimate.  Consider a random sample $s\subseteq U$ where $G_U$ follows the \er model with edge probability $\pi$. We observe the nodes $i\in s$, as well as network degrees $d_i^s:=\sum_{j\in s} \indicator{E_{ij}=1}$ and $d_i^U:=\sum_{j\in U}\indicator{E_{ij}=1}$, for each individual $i\in s$. 
Then, $\E\l \sum_{i\in s} d_i^U\r=2{n\choose 2}\pi$ and $\E\l \sum_{i\in s} d_i^s \r=\pi n(N-1)$.
Canceling out $\pi$ yields the MME, 
which is often simplified to
\begin{equation}
\wh N=\frac{n\sum_{i=1}^n d_i^U}{\sum_{i=1}^nd_i^s}.
\label{nsum}
\end{equation}

\citet{chen2016estimating} investigated the behavior of $\wh N$ with finite-sample as well as with large $n$, but did not specify the relationship between $N$ and $n$ under the asymptotic setting. In our setting, the finite-population regime is $n\rightarrow N$ with $N$ fixed. The infill regime is that $n, N$ are fixed and the sampling procedure is infinitely repeated. The outfill asymptotic regime is that $n_t, N_t\rightarrow\infty$ with $n_t/N_t\rightarrow c\in (0,1)$. Then we have the following theorem for the asymptotic properties of $\wh N$. 

\begin{theorem}
Under the finite-population regime, $\wh N$ converges to $N$. Under infill asymptotics, $\wh N$ is always positively biased conditioning on $|E_s|>0$ \citep{chen2016estimating}, and is hence inconsistent.  Under outfill asymptotics, $\wh N$ is asymptotically normal with bias $(1-c)/c$ and variance $O(1)$.
\label{nsum_asym2}
\end{theorem}

\subsection{Estimating a total with unequal sampling probabilities}

A generalization of binomial models allows for heterogeneity in the inclusion, or ``success'' probabilities $p$, that is, when the sampling is not uniformly at random. \citet{horvitz1952generalization} proposed unbiased estimators for population means and totals under the setting of sampling without replacement from finite population, where the selection probabilities can be unequal.  
 The Horvitz-Thompson (HT) estimator for the population total is $\wh N =  \sum_{i\in s}{1}/{p_i}$, 
where $p_i = \E(\indicator{i\in s})$ is the probability that unit $i\in U$ is sampled in $s$.  The estimator $\wh N$ is unbiased for the total population size $N$. 
This estimator and its variants have been applied to the estimation of animal abundance \citep{borchers1998horvitz} and other fields.
We consider a deterministic sample size $n$. Then the variance of $\wh N$ is \citep{horvitz1952generalization}
\begin{equation}
\v(\wh N)=-\frac{1}{2}\sum_{i=1}^N\sum_{j=1}^N(p_{ij}-p_i p_j)\l \frac{1}{p_i}-\frac{1}{p_j} \r^2,
\label{varHT2}
\end{equation}
where $p_{ij}$ is the joint probability that units $i$ and $j$ are both in the sampled set $s$, and $p_{ii}=p_i$. The finite-population regime amounts to letting $p_i\rightarrow 1$ for any $i$. Under the infill regime, $p_i, p_{ij}, N$ are fixed and the number of repeated samples $k_t\rightarrow\infty$. Under the outfill regime, $N$ and $n$ both increase to infinity such that $n/N\rightarrow c\in (0,1)$. Figure \ref{fig:ht} shows the non-uniform sampling mechanism under each regime.

\begin{figure} 
  \centering 
  \includegraphics[width=12cm]{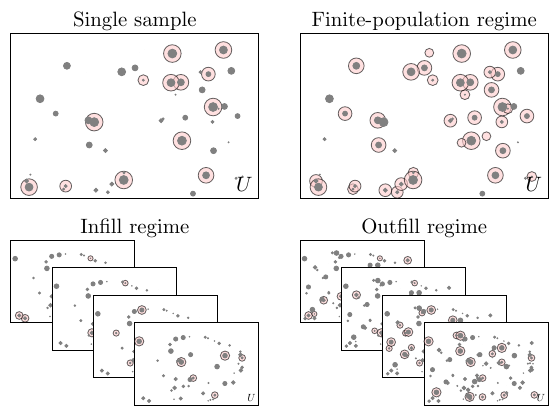}
  \caption{Illustration of the single sample, finite-population, infill and outfill regimes for the general HT estimator. The probability of being sampled for each point here is visualized as its size. }
  \label{fig:ht}
\end{figure}

Specifically, we consider the following setting to illustrate the asymptotic behavior of the HT estimator. Suppose $U$ consists of $H$ clusters, where the $h$th cluster has $N_h$ units.  We assume that $H$ is known in advance, while $N_h$ is observed only if a unit from cluster $h$ is sampled. In each sample, a total of $n$ units are sampled from $U$ by the following procedure: first a cluster $h$ is drawn uniformly at random each with probability $1/H$. Then one unit is drawn from the $N_h$ units in that cluster, also uniformly at random, without replacement. We assume that $\min_{h \in [H]}N_h>n$. An observation on sample $s$ consists of the units in $s$, their cluster membership, and the sizes of clusters that they belong to. 


When there are repeated observations, we assume they follow the same design and are mutually independent. 
In this setting, the outfill regime is defined such that each cluster in the original population is replicated and appears $t$ times in $U_t$. 
The cluster sizes are fixed at $N_h^{(t)}= N_h$ and the number of clusters increases as $H_t=tH$.  $N=\sum_{h=1}^H$ is fixed and the estimand is $N_t=Nt$.
The sample size satisfies $n_t/N_t\rightarrow c\in (0,1)$. We then have the following theorem about the consistency of $\wh N$ under each regime. In particular, the variance of $\wh N$ grows with $N$ under the outfill regime.

\begin{theorem}
$\wh N$ is consistent under the finite-population regime, and MSE consistent under infill asymptotics. 
$\wh N$ is unbiased and asymptotically normal with variance $O(N)$ under the outfill regime.
\label{thm:ht}
\end{theorem}

\section{Other unordered sets}
\label{sec:unordered}

\subsection{Capture-recapture experiments}

Capture-recapture (CRC) refers to a broad class of methods to estimate the size of hidden populations 
for which random sampling is possible \citep{seber1973estimation, fienberg1972multiple, darroch1958multiple, pollock1982capture, chao1987cap, jolly1965explicit}. Estimation of the population size is based on the overlap between two or more random samples \citep{hook1995capture, robles1988application, van2015estimating, paz2011how}. While a wide variety of CRC estimators have been developed \citep{schwarz1996general, young1998capture, pollock1982capture, chao1987cap, khan2017one}, we focus here on the two- and $k$-sample CRC estimators with homogeneity within a closed population.

\subsubsection{Two-sample estimation}
\label{subsubsec:2_samp_crc}

We first consider the common case of two-sample CRC. Let $U$ be a hidden finite set of size $N$, where each unit $i\in U$ has binary attributes $(X^1_i, X^2_i)$, which are all $(0,0)$ in the beginning. We draw a sample $s_1\subseteq U$ with size $n_1$ from $U$, and set $X^1_i=1$ for all $i\in s_1$. Then a second sample $s_2$ with size $n_2$ is drawn, independent from $s_1$ and uniformly at random, and we set $X^2_i=1$ for all $i\in s_2$. We observe $(X^1_i, X^2_i)_{i\in s_1\cup s_2}$, and let $m=\sum_{i\in U}\indicator{(X^1_i, X^2_i)=(1,1)}$.
In ecology, to estimate the abundance of an animal species, researchers could first capture $n_1$ animals from that species, mark them and then release them. After the captured animals have mixed well with the remaining ones, researchers could capture $n_2$ animals again, uniformly at random, and record the number $m$ of animals captured in the first step.
Then $m$ follows a hypergeometric distribution conditioning on $N, n_1$ and $n_2$, i.e. the mechanism of generating the observations can be defined as $\P(m|s_1, s_2)={n_1\choose m}{N-n_1\choose n_2-m}/{N\choose n_2}$. The MME, $\wh N_L={n_1n_2}/{m}$, is also known as the Lincoln-Petersen estimator \citep{lincoln1930calculating, petersen1894biology}. 

We consider the finite-population regime with $n_2\rightarrow N$. The infill regime is that $N, n_1, n_2$ are fixed and repeated sample pairs $\{s_1^{(t)}, s_2^{(t)}\}$ are drawn with $t\rightarrow\infty$. The outfill regime is given by $N^{(t)}, n_1^{(t)}, n_2^{(t)}\rightarrow\infty$ with $n_i^{(t)}/N^{(t)}\rightarrow c_i\in (0,1)$ for $i=1,2$.

Previous results exist on the bounds or estimates of biases and variances. These were implicitly based on asymptotic approximations: \citet{chapman1951some} showed a lower bound for the bias
\begin{equation*}
\E\l\wh N_L\r-N \ge N\left[ \frac{N}{n_1n_2}+2\l\frac{N}{n_1n_2}\r^2 \right]
\end{equation*}
under outfill, 
and bounded the variance as
\begin{equation*}
{\v}(\wh N_L)> N^2 \left[\l\frac{N}{n_1n_2}\r+\l\frac{N}{n_1n_2}\r^2\right]
\end{equation*}
under asymptotic approximation that was satisfied by the outfill regime. Though these no longer hold under finite-sample setting, calculations in \citep{chapman1951some} showed that $\wh N_L$ has a considerable bias under a range of settings. A less biased estimator
\begin{equation}
\wh N_C=\frac{(n_1+1)(n_2+1)}{m+1}-1,
\label{chap_est}
\end{equation} 
was proposed \citep{chapman1951some}, with bias

\begin{equation}
\E(\wh N_C)-N=-\frac{(N-n_1)!(N-n_2)!}{N!(N-n_1-n_2-1)!}
\label{chap_bias}
\end{equation}

for any $n_1, n_2, N$, and variance
\begin{equation}
\v(\wh N_C)\sim N^2\left[\frac{N}{n_1n_2}+2\l\frac{N}{n_1n_2}\r^2+6\l\frac{N}{n_1n_2}\r^3\right]
\label{chap_var}
\end{equation}
under outfill \citep{chapman1951some}, where $\sim$ means the difference between two quantities decay to 0. We have the following asymptotic result of $\wh N_L$ and $\wh N_C$. 
Specifically, both estimators have infinitely increasing estimation error under the outfill asymptotic setting.

\begin{theorem}
Under the finite-population regime, $\wh N_L$ and $\wh N_C$ are consistent. 
Under infill asymptotics, $\wh N_L$ is positively biased and has MSE $O(1)$ for at least a range of values of $n_1, n_2, N$. $\wh N_C$ is negatively biased, but the bias is within 1 if $n_1+n_2+1<N/2$ and $n_1n_2/N>\log N$ \citep{chapman1951some}. 
Under the outfill regime, $\wh N_L$ has bias at least $O(1)$ and variance at least $O(N)$. $\wh N_C$ is asymptotically unbiased with variance $O(N)$. 
Furthermore, $\wh N_C$ and $\wh N_L$ are inconsistent with $\P(|\wh N_{C}-N|<\varepsilon)\rightarrow 0$ and $\P(|\wh N_{L}-N|<\varepsilon)\rightarrow 0$ for some $\varepsilon>0$ when $n_1=c_1N, n_2=c_2N$.
\label{chapman}
\end{theorem}

Further, \citet{chapman1951some} showed that no estimator can be unbiased for all possible values of $N, n_1$ and $n_2$.

A similar but slightly different sampling mechanism gives rise to the multiplier method, also called the method of benchmark multiplier (MBM). 
In practice, researchers may know the number of hidden units with a certain trait. The overall prevalence of that trait in the hidden population, if available from estimation, would provide an estimate for the size of the hidden population. Often the prevalence is estimated through expert opinion, historical data, or from a separate sample \citep{godfrey2002economic, frischer2001comparison, hickman2006estimating}.

We consider the last approach. The idea of MBM can be expressed with a sampling mechanism similar to CRC, except that the first sample $s_1$ is fixed under infill asymptotics. That is, the known sub-population of hidden units with a certain trait is fixed. The size $n_1$ of $s_1$ is called the \emph{benchmark}.
 The proportion $m/n_2$ gives the \emph{multiplier}, which is an estimate of the prevalence $p$. Again, $m$ follows a hypergeometric distribution, 
 so the MME for $N$ is $\wh N_{MBM}=n_1n_2/{m}$, which is often called the multiplier estimator. $\wh N_{MBM}$ takes the same form as the Lincoln-Petersen CRC estimator. Asymptotic behaviors of $\wh N_{MBM}$, as summarized in Theorem \ref{asym_mbm}, are essentially the same as that of $\wh N_L$ for CRC.

\begin{theorem}
  $\wh N_{MBM}$ is consistent under the finite-population regime. 
  Under infill asymptotics, $\wh N_{MBM}$ is inconsistent with MSE $O(1)$.
  Under the outfill regime, when $n_1=c_1N$, and $n_2=c_2N$, $\wh N_{MBM}$ is inconsistent with MSE at least $O(N)$. $\P(|\wh N_{MBM}-N|<\varepsilon)\rightarrow 0$ for some $\varepsilon>0$.
\label{asym_mbm}
\end{theorem}

\subsubsection{$\bm k$-sample estimation}

We now consider the generalized setting of $k$ samples. In this scenario, we draw $k$ samples $s_1, \ldots, s_k\subseteq U$ with deterministic sizes $n_1, \ldots, n_k$ respectively. We assume the probability $p_j:=n_j/N$ of being observed in the $j$th sample is the same for each unit for $j=1, \ldots, k$. In each sample (say $s_j$), we give the observed units a label that is different for different $j$'s, and record the capture history $\mathcal H_{j,i}=(I_1^{(i)}, \ldots, I_j^{(i)})$ of each unit $i \in s_j$, where $I_l=1$ if $i\in s_l$ and 0 otherwise $(l\le j)$. Then an observation on a sequence of samples $\bm s=\{s_1, \ldots, s_k\}$ is a $2^k$ contingency table $T=\{T_{I_1\ldots I_k}\}_{I_1, \ldots, I_k \in \{0,1\}^{k}}$ \citep{fienberg1972multiple}, where the entry corresponding to $I_1, \ldots, I_k$ is $\sum_{i\in U} \mathbbm{1}(I_1^{(i)}=I_1, \ldots, I_k^{(i)}=I_k)$, the number of units with capture history $\mathcal H_k=(I_1, \ldots, I_k)$. Let $r$ be the sum of known entries in the contingency table -- only the entry $T_{0\ldots 0}$ is unobserved. 
In plain words, following the animal abundance example, researchers could instead draw $k$ random samples. In the first $k-1$ samples, animals that are captured will be given a mark that is unique for each sample. The contingency table summarizes the capture history for all observed animals -- how many animal(s) are observed in, or absent from, which sample(s).
From the contingency table we have $m_i$, the number of already marked individuals in $s_i$, and $M_i$, the total number of marked individuals in $U$ before $s_i$ is drawn. 
The sampling scheme then follows a generalized hypergeometric distribution:
\begin{equation}
\P(T|s_1, \ldots, s_k)=\frac{N!}{\prod_{I_1, \ldots, I_k \in \{0,1\}^{k}} T_{I_1\ldots I_k}!(N-r)!}\prod_{i=1}^k {N \choose {n_i}}^{-1}.
\label{prob:CRC}
\end{equation}

Maximizing the likelihood \eqref{prob:CRC} gives the MLE of $N$ as the solution of 
\begin{equation}
\l 1-\frac{r}{N} \r=\prod_{i=1}^k \l 1-\frac{n_i}{N}\r,
\label{equ:darroch}
\end{equation}
which is unique, finite and greater than $r$ if $s_1\cap \ldots \cap s_k$ is non-empty and $|s_i|<r$ for all $i\le k$ \citep{darroch1958multiple}. We restrict our interest to this case only. Setting $k=2$ recovers the Lincoln-Petersen estimator $\wh N_L$. Since finite-population and infill regimes for the two- and $k$-sample cases are similar in essence, we mainly discuss outfill asymptotics in this setting: for any finite $k$, we have $N, n_1, \ldots, n_k\rightarrow\infty$ with $n_i/N_i\rightarrow c_i\in (0,1)$ for $i=1, \ldots, k$, and $k_t$ may be finite or going to infinity. We assume the $c_i$'s are bounded away from 0 and 1. 
 Under outfill asymptotics with finite $k$, following from the delta method, the bias of the MLE is approximated by \citep{darroch1958multiple}
$$\E \l\wh N_{MLE}\r-N\sim\frac{\left[\frac{k-1}{N}-\sum\l \frac{1}{ N-n_i} \r\right]^2+\left[\frac{k-1}{ N^2}-\sum\l \frac{1}{ N-n_i} \r^2\right]}{2\left[ \frac{1}{ N-\E [r]}+\frac{k-1}{ N}-\sum\l \frac{1}{ N-n_i} \r \right]^2},$$
which is $O(1)$, and the asymptotic variance is $O(N)$, approximated by \citep{darroch1958multiple}
$${\v}(\wh N_{MLE})\sim\left[ \frac{1}{N-\E [r]}+\frac{k-1}{ N}-\sum_{i=1}^k \l \frac{1}{ N-n_i} \r \right]^{-1}.$$
Under outfill asymptotics with infinite sampling repetitions, we assume $\inf_{i\in[k]} p_i>0$. Then the magnitude of bias is bounded above by $N-\E [r]$, and hence by $N\prod_{i=1}^k (1-p_i)$. The variance is $O(N\prod_{i=1}^k (1-p_i))$. Therefore, as long as $k$ is increasing such that $N\prod_{i=1}^k (1-p_i)\rightarrow 0$, $\wh N_{MLE}$ will be MSE consistent for $N$. 





\section{Discussion}

Several features determine researchers' ability to learn about the size of a hidden set.  First, the structure of the set -- labeled units, ordering of the labels, or relational (network/graph) information -- can permit researchers to learn about the number of remaining units when a subset is observed.  Second, a feasible probabilistic query mechanism -- random sampling, or observation conditional on a unit trait or attribute -- must be available.  Third, a statistical estimator that enjoys desirable statistical properties must be chosen.  Some of these features may be under the control of researchers, while others may be intrinsic to the problem.  
Table \ref{tab:summary} summarizes the models that have been discussed in this paper, as well as consistency results of estimators in each model.

\begin{table}
\small
\begin{tabular}{lcccc}
\toprule
Problem  & Trait  & Sample & Estimator & Consistency \\ 
 \midrule
German tank  & \makecell{Consecutive\\ integers} & \makecell{Uniform random\\ draw from $U$}  & $\wh N_\text{MLE},\wh N_G,\wh N_2,\wh N_3$ & \makecell{Infill: consistent \\Outfill: MSE $O(1)$} \\ 
Binomial $N$ & $Y_i\sim$Bernoulli($p$) & $\{i\in U: Y_i=1\}$ & \makecell{$\wh N_\text{MLE}, \wh N_\text{MME}$, \\ $Q_{(k)}, \wh N'_\text{MLE}$ } & \makecell{Infill: consistent \\ Outfill: $\wh N_\text{MME}, \wh N'_\text{MLE}$ \\have MSE $O(1)$$^1$}\\
Waiting times & $T_i\sim$Exponential$(\lambda)$  & $\{i\in U: T_i<T\}$ & \multicolumn{2}{c}{Equivalent to Binomial $N$}\\
\multirow{2}{*}{NSUM}& \makecell{Network degree\\$d_i^U, d_i^V$} & \makecell{Uniform random\\draw from $V\setminus U$} & $\wh N_{NS}=M\frac{\sum d_i^U}{\sum d_i^V}$ & \makecell{Infill: inconsistent\\Outfill: MSE $O(1)$}\\
& \makecell{Network degree\\$d_i^s, d_i^U$} & \makecell{Uniform random\\draw from $U$} & $\wh N=n\frac{\sum d_i^U}{\sum d_i^s}$ & \makecell{Infill: inconsistent$^2$ \\Outfill: MSE $O(1)$}\\
HT & Cluster membership & \makecell{Uniform random\\draw from each\\sampled cluster} & $\wh N=\sum 1/p_i$ & \makecell{Infill: consistent\\Outfill: MSE $O(N)$}\\
$k$-sample CRC & Capture history & $\P(i\in s_j)=n_j/N, \forall i$ & $\wh N_L, \wh N_C, \wh N_{MLE}$ & \makecell{Infill: inconsistent\\Outfill $k=2$: MSE $O(N)$\\ Outfill $k\rightarrow\infty$: $\wh N_{MLE}$ \\consistent$^3$}\\
MBM & \multicolumn{4}{c}{Similar to two-sample CRC}\\
\bottomrule
\end{tabular}
\caption{Summary of models, estimators, and asymptotic results for estimating the size of a hidden set. Notes: 1. the error rate of $\wh N_{MLE}$ and $Q_{(k)}$ are given in Section \ref{subsec:binom} in terms of relative error; 2. conditioning on $|E_s|>0$; 3. if $N\prod_{i=1}^k(1-p_i)\rightarrow 0$. }
\label{tab:summary}
\end{table}

How should empirical researchers evaluate the statistical properties of estimators, design a study or choose a sample size? Many of these tasks are based on asymptotic arguments, and statistical claims about the large-sample performance of hidden set size estimators depend on specification of an appropriate asymptotic (or even non-asymptotic) regime. 
It is crucial to identify how the sample size increases, especially in relation to the target population, when asymptotic approximation or comparison is involved in population size estimation tasks. When designing a study, this may include determining the minimum sample size that leads to a desired precision \citep{robson1964sample, jensen1981sample}, or selecting an ``optimal'' sampling strategy (e.g. one-time larger sample versus multi-time repeated smaller samples). In data analysis, this may include establishing valid approximation to biases and variances or comparing the efficiency of different statistical approaches \citep{robson1964sample, bailey1951estimating, brownie1985analysis, mills2000estimating}. 
If the vast majority of the target population can be observed in one-step sampling, consistency under the trivial finite-population regime may be a goal when developing estimators. 
If the total population is fixed, and arbitrarily repeated i.i.d. samples can be obtained, then consistency under infill may justify the use of a statistical approach. If instead only one-time or finite-time sampling is permitted, in which the sample size is believed to reflect a proportion of the potentially large population, performance of estimators under outfill may be of more interest. 
We have shown that different asymptotic regimes can lead to dramatically different statistical properties. Some seemingly sensible estimators are inconsistent with different rates of MSE, and asymptotic claims for population size estimators under one regime may be of limited value for analyzing the general situation. 



In this review, we have focused on technical claims about the asymptotic properties of estimators, and have not discussed considerations for practical data collection.  For example, the waiting time model does not accommodate censoring or truncation of observations, but could be easily extended to do so.  Respondent recall bias in the network scale-up method may make the reported network degrees noisy estimates of the truth. The Horvitz-Thompson estimator relies on knowledge about marginal inclusion probabilities of each sampled individual, which may not be readily available when the size of the population is unknown.  While improved data collection strategies may not be able to mitigate poor asymptotic properties -- like inconsistency -- under a particular regime, better data may be able to reduce variance in finite samples.  

While we have discussed many of the most popular settings and methods for estimating the size of a hidden set, there are several other settings we have not covered. Respondent-driven sampling (RDS), snowball sampling and link-tracing sampling generate samples from hidden networks, and modeling the stochastic process underlying such sampling mechanism can be used to estimate hidden population sizes \citep{handcock2014estimating, crawford2017hidden,vincent2014estimating,vincent2016estimating}.  There is a large literature on CRC beyond what we have covered here. For example, there are approaches for CRC with an open population, with immigration, emigration, birth, and death \citep{schwarz1996general, young1998capture} or with heterogeneity in capture probabilities \citep{pollock1982capture, chao1987cap}.  CRC is also possible using data from network sampling designs \citep{khan2017one}. 
We have also not discussed species number estimation \citep{bunge2014estimating}, ``count distinct'' and streaming estimation problems \citep{fusy2007estimating,chassaing2006,Kane2010}, and genetic methods for population size estimation \citep{creel2003population,bellemain2005estimating}. In addition, we have not addressed the issue of entity resolution, or record de-duplication \citep{sadosky2015blocking}. The results presented in this paper suggest that researchers employing methods for estimating the size of a hidden set should evaluate the performance of estimators under deliberately specified asymptotic assumptions.  






\section*{Acknowledgements}

This work was supported by NIH grants NICHD/BD2K DP2 HD091799, NIH/NCATS KL2TR000140, NIMH P30MH062294, the Center for Interdisciplinary Research on AIDS, and the Yale Center for Clinical Investigation. 
We are grateful to 
Peter M. Aronow, 
Heng Chen,
Xi Fu,
and
Edward H. Kaplan
for helpful comments. 


\appendix

\section{Asymptotic normality and consistency}
We first introduce a simple lemma that helps to prove consistency or inconsistency based on asymptotic normality. 

\begin{lemma}[Asymptotic normality and consistency]
Suppose $a_t\l X_t-\nu_t \r\xrightarrow{L} N(0, \sigma^2)$ for a finite $\sigma$. Then $X_t-\nu_t\xrightarrow{P} 0$ if and only if $a_t\rightarrow\infty$ as $t\rightarrow\infty$.
\label{norm}
\end{lemma}

\begin{proof}
Assume $a_t\rightarrow\infty$ when $t\rightarrow\infty$. Then for any $\varepsilon>0$, there exists $T>0$ such that $\lvert\E(X_t-\nu_t)\rvert<\varepsilon/2$ for any $t>T$. 
For such $t$, since $|X_t-\nu_t|\le \lvert X_t-\nu_t-\E(X_t-\nu_t) \rvert + \lvert\E(X_t-\nu_t)\rvert$, applying the union bound and Chebyshev's inequality yields
\begin{align*}
&\P\left\{ \lvert X_t-\nu_t \rvert>\varepsilon \right\} =\P\left\{ \lvert X_t-\nu_t-\E(X_t-\nu_t) + \E(X_t-\nu_t)\rvert>\varepsilon \right\} \\
&\qquad \le\P\left\{ \lvert X_t-\nu_t-\E(X_t-\nu_t) \rvert >\frac{\varepsilon}{2} \right\}+\P\left\{ \lvert\E(X_t-\nu_t)\rvert>\frac{\varepsilon}{2} \right\} \le \l \frac{2\sigma}{\varepsilon a_t}  \r ^2 \rightarrow 0.
\end{align*}

If $a_t$ does not go to infinity, then for some $m>0$ and any $M>0$, there exists $t>M$ such that $a_t<m$ for such $t$. Pick $\varepsilon>0$, then there exists $T>0$ such that $\P\left\{ a_t(X_t-\nu_t)>\varepsilon m \right\}\ge 1/2\cdot \P\left\{ Y>\varepsilon m \right\}$ for $Y\sim N(0, \sigma^2)$ for all $t>T$. Specially, for any $M>0$, there exists $t_0>\max\{T, M\}$ such that $a_{t_0}<m$ holds. Then
\begin{align*}
\P\left\{ \left| X_{t_0}-\nu_{t_0} \right|>\varepsilon \right\}& \ge \P\left\{ a_{t_0}(X_{t_0}-\nu_{t_0}) > \varepsilon m\right\} \ge \frac{1}{2}\P(Y>\varepsilon m)>0,
\end{align*}
indicating that $\{X_t-\nu_t\}$ does not converge to 0.
\end{proof}

\section{Proof of theorems}

\begin{proof}[Proof of Theorem \ref{thm_tank}] 
Rates of biases and variances of $\wh N_{G}, \wh N_2$ and $\wh N_3$ follow from the non-asymptotic claims of biases and variances given by \citet{goodman1952serial}, as stated in the main text. Consistency under the finite-population regime follows directly from setting $n= N$ in each of the estimators. Consistency under infill of $\wh N_G, \wh N_2, \wh N_3$ follows from the unbiasedness of these estimators, while that of $\wh N_{MLE}$ follows from the fact that $\max_{i\in [k_t]}X_{i(n)}\xrightarrow{P} N$ as $k_t\rightarrow \infty$.

We show the inconsistency of $\wh N_G$ (when the initial label is 1) and $\wh N_3$ (when the initial label is unknown) under outfill. The results can be derived similarly for $\wh N_{MLE}$ and $\wh N_2$, since they are shifted and scaled versions of $\wh N_G$, and the corresponding proofs for inconsistency also amount to bounding the probability that $X_{(n)}$ equals a specific value (as done below).

For $\wh N_G$, recall that
$$\wh N_G=\frac{n+1}{n}X_{(n)}-1,$$
where $\wh N_G, n$ and $N$ are implicitly indexed by $t$ as defined under outfill asymptotics. However, we omit the subscript for simpler notation.
For any $0<\varepsilon<1/c-1$, there exists $T_0\in N_+$ such that
\begin{equation}
\frac{N}{n+1}\le \frac{1}{c}+\frac{\varepsilon}{2}, \quad \text{and} \quad \frac{n}{n+1}\ge 1-\frac{\varepsilon}{2}
\label{pf411}
\end{equation}
for any $t>T_0$, where $n, N$ are indexed by $t$. Then when $t>T_0$, 
\begingroup
\allowdisplaybreaks
\begin{align}
&\notag\P(\wh N_G=N) = \P\left( \frac{n+1}{n}X_{(n)}-1=N \right)\\
\notag &\qquad= \P\left( X_{(n)}-N=-\frac{N}{n+1}+\frac{n}{n+1} \right)\\
\notag &\qquad\le \P\left( X_{(n)}\ge N-\frac{1}{c}+1-\varepsilon \right)\\
\notag &\qquad\le 1-\P\left( X_{(n)}=\lceil N+\frac{c-1}{c}-\varepsilon \rceil-1 \right)\\
\notag &\qquad=1-{N\choose n}^{-1}{\lceil N+\frac{c-1}{c}-\varepsilon\rceil -2 \choose n-1}\\
&\qquad\le 1-\frac{n}{N}\left( \frac{\lceil N+\frac{c-1}{c}-\varepsilon\rceil-n}{N-n+1} \right)^{n-1}
  \rightarrow 1-\exp\left( -\frac{c}{1-c} \right)<1.
\label{pf412}
\end{align}
\endgroup

\par Then we show the inconsistency of $\wh N_3$ with unknown initial number $u$. Let $Y=X-u$, then $Y_{(n)}-Y_{(1)}$ and $X_{(n)}-X_{(1)}$ follow the same distribution. Likewise, for any $\varepsilon>0$, there exists $T_1$ such that
\begin{equation}
\frac{n-1}{n+1}\ge 1-\frac{\varepsilon}{2}, \quad\text{and}\quad -\frac{2N}{n+1}\ge -\frac{2}{c}-\frac{\varepsilon}{2}
\label{pf413}
\end{equation}
for any $t>T_1$. Then
\begin{align}
\notag&\P\left(\wh N_3=N\right) =\P\left( Y_{(n)}-Y_{(1)}-N=\frac{n-1-2N}{n+1} \right)\\
\notag&\qquad\le \P\left(Y_{(n)}-Y_{(1)}\ge N+1-\frac{2}{c}-\varepsilon\right)\\
\notag&\qquad\le 1-\P\left(Y_{(n)}=\lceil{N+1-\frac{2}{c}-\varepsilon}\rceil-1\right)\\
\notag&\qquad=1-{N\choose n}^{-1}{\lceil{N+1-\frac{2}{c}-\varepsilon}\rceil-2 \choose n-1}\\
&\qquad\le 1-\frac{n}{N}\left(\frac{\lceil{N+1-\frac{2}{c}-\varepsilon}\rceil-n}{N-n+1} \right)^{n-1} 
  \rightarrow 1-\exp\left(-\frac{c}{1-c}\right)<1.
\label{pf414}
\end{align}
Since $\wh N_G$ and $\wh N_3$ take discrete values, (\ref{pf412}) and (\ref{pf414}) imply the inconsistency of $\wh N_G$ and $\wh N_3$.
\end{proof}

\begin{proof} [Proof of Theorem \ref{nsum_asym1}]
We first show the conditional distribution of $d_i^U$ given $d_i^V$. Recall that $\pi$ is the edge probability $\P(\{i,j\}\in E)$ for any $i\in V\setminus U$ and $j\in V$. Then, 
\begin{align*}
\P(d_i^U=k\mid d_i^V) & =\frac{\P(d_i^U=k, d_i^V)}{\P(d_i^V)}=\frac{{N\choose k}\pi^k(1-\pi)^{N-k}{M-N-1\choose d_i^V-k}\pi^{d_i^V-k}(1-\pi)^{M-N-1-d_i^V+k}}{{M-1\choose d_i^V}\pi^{d_i^V}(1-\pi)^{M-1-d_i^V}}\\
&=\frac{{N\choose k}{M-N-1\choose d_i^V-k}}{{M-1\choose d_i^V}},
\end{align*}
$d_i^U$ follows a hypergeometric distribution given $d_i^V$. Therefore,
\[
\E\l \wh N_{NS} \mid d_i^V, i=1, \ldots, n \r=M\frac{\sum d_i^V\cdot \frac{N}{M-1}}{\sum d_i^V}=\frac{MN}{M-1},
\]
and $\E(\wh N_{NS})-N=N/(M-1)$.

Since we impose no assumption on the distribution of network degrees within $U$, even when we sample all units in $V\setminus U$, we cannot recover $N$ deterministically. (For example, when there exists $j\in U$ such that $d_j^{V\setminus U}=0$.)
Under infill asymptotics, repeated i.i.d. samples are taken and the estimates are averaged. The final estimate therefore converges to a quantity with constant bias $N/(M-1)$.
We now derive the asymptotic distribution of $\wh N_{NS}$ under outfill, where $N/M\rightarrow c_1$ and $n/(M-N)\rightarrow c_2$. 
First, 
\begin{equation}
\frac{\sum d_i^V}{nM}=\frac{\sum d_i^s}{n(n-1)}\frac{n-1}{M}+\frac{\sum d_i^{V\setminus s}}{n(M-n)}\frac{M-n}{M},
\label{decomp1}
\end{equation}
where $\sum d_i^s/2\sim$ Binomial$\l{n\choose 2}, \pi\r$, and $\sum d_i^{V\setminus s}\sim$ Binomial$\l n(M-n), \pi \r$. By the central limit theorem and Slutsky's theorem,
\begin{equation}
\sqrt{n(n-1)}\l \frac{\sum d_i^s}{n(n-1)}-\pi \r\xrightarrow{L}N(0,2\pi(1-\pi)),
\label{clt1}
\end{equation}
\begin{equation}
\sqrt{(M-n)n}\l \frac{\sum d_i^{V\setminus s}}{(M-n)n}-\pi \r\xrightarrow{L} N(0, \pi(1-\pi)).
\label{clt2}
\end{equation}
Multiply (\ref{clt1}-\ref{clt2}) by $(n-1)/\sqrt{n(n-1)}$ and $\sqrt{(M-n)/n}$ respectively and by Slutsky's theorem we have
\begin{equation}
M\l \frac{\sum d_i^s}{n(n-1)}\cdot\frac{n-1}{M}-\frac{\pi(n-1)}{M} \r\xrightarrow{L} N(0, 2\pi(1-\pi)),
\label{clt11}
\end{equation}
\begin{equation}
M\l \frac{\sum d_i^{V\setminus s}}{(M-n)n}\cdot\frac{M-n}{M}-\frac{\pi(M-n)}{M} \r\xrightarrow{L} N\l 0, \pi(1-\pi)\frac{1-c_2(1-c_1)}{c_2(1-c_1)} \r.
\label{clt21}
\end{equation}
Since $\sum d_i^s$ and $\sum d_i^{V\setminus s}$ are mutually independent, combining (\ref{decomp1}), (\ref{clt11}) and (\ref{clt21}) yields
\begin{equation}
\sqrt{nN}\l \frac{\sum d_i^V}{nM}+\frac{\pi}{M}-\pi \r\xrightarrow{L} N\l 0, \pi(1-\pi)c_1[1+c_2(1-c_1)]\r.
\label{dist1}
\end{equation}

Also,
\begin{equation*}
\sqrt{nN}\l \frac{\sum d_i^U}{nN}-\pi \r\xrightarrow{L}N(0, \pi(1-\pi)).
\end{equation*}
Divide both sides by $\sum d_i^V/nM$, and Slutsky's theorem yields
\begin{equation*}
\frac{\frac{\sum d_i^U}{n}\cdot\frac{\sqrt{nN}}{N}}{\sum d_i^V/(nM)}-\frac{\sqrt{nN}\pi}{\sum d_i^V/(nM)}\xrightarrow{L} N\l 0, \frac{1-\pi}{\pi}\r,
\end{equation*}
which can be rewritten as
\begin{equation}
\frac{\sqrt{nN}}{N}\l \wh N-N \r+\sqrt{nN}\l 1-\frac{\pi}{\sum d_i^V/nM} \r\xrightarrow{L} N\l 0, \frac{1-\pi}{\pi} \r.
\label{decomp0}
\end{equation}

Learning about the asymptotic behavior of $\wh N-N$ requires characterizing the second term on the left-hand side of (\ref{decomp0}). Define a sequence of random variables and functions
\[
X_t=\frac{\sum d_i^V}{nM}+\frac{\pi}{M} \quad \text{and} \quad g_t(x)=1-\frac{\pi}{x-\frac{\pi}{M}},
\]
where $n, M$ are indexed by $t$, and a function $g(x)=1-\pi/x$. Then 
\begin{equation}
\sqrt{nN}\l 1-\frac{\pi}{\sum d_i^V/nM}\r=\sqrt{nN} g_t(X_t)=\sqrt{nN} [g_t(X_t)-g(X_t)]+\sqrt{nN} [g(X_t)-g(\pi)]
\label{decomp2}
\end{equation}
since $g(\pi)=0$. The first term in (\ref{decomp2}) satisfies
\begin{align*}
\sqrt{nN} [g_t(X_t)-g(X_t)] & =\sqrt{nN}\l \frac{\pi}{\sum d_i^V/nM+\pi/M}-\frac{\pi}{\sum d_i^V/nM} \r\\
&=-\frac{\pi^2\sqrt{nN}}{M}\cdot\frac{1}{\frac{\sum d_i^V}{nM}\l \frac{\sum d_i^V}{nM}+\frac{\pi}{M} \r}\xrightarrow{P} -\sqrt{c_1c_2(1-c_2)},
\end{align*}
and the second term in (\ref{decomp2}) satisfies
\[
\sqrt{nN}[g(X_t)-g(\pi)]\xrightarrow{L} N\l 0, [g'(\pi)]^2\pi(1-\pi)c_1[1+c_2(1-c_1)] \r
\]
by the delta method. Therefore the quantity in (\ref{decomp2})
\begin{equation}
\sqrt{nN}\l 1-\frac{\pi}{\sum d_i^V/nM}\r\xrightarrow{L} N\l -\sqrt{c_1c_2(1-c_2)}, \frac{(1-\pi)c_1[1+c_2(1-c_1)]}{\pi} \r
\label{equ1}
\end{equation}
by Slutsky's theorem. Combining (\ref{decomp0}) and (\ref{decomp2}), we have
$
 \wh N-N\xrightarrow{L} N(c_1, \sigma^2),
$
where $\sigma^2$ is bounded between,
$$
  \frac{(1-\pi)c_1}{\pi c_2(1-c_1)}\left[ 1+c_1(1+c_2(1-c_1))\pm 2\sqrt{c_1(1+c_2(1-c_1))} \right].
$$
Therefore, $\wh N$ is asymptotically normal with bias $c_1$ and variance $O(1)$, and following from Lemma \ref{norm}, inconsistent under the outfill regime.
\end{proof}

\begin{proof} [Proof of Theorem \ref{nsum_asym2}]
We derive the asymptotic normal distribution of $\wh N$ under the outfill regime that $n/N\rightarrow c\in (0,1)$. Note that

\begin{equation*}
\wh N=n+n\frac{\sum d_i^{U\setminus s}}{\sum d_i^s},
\end{equation*}
and $\sum d_i^s/n(n-1)\xrightarrow{P} \pi$. Also,
\begin{equation}
\sqrt{n(N-n)}\l \frac{\sum d_i^{U\setminus s}}{(N-n)n}-\pi \r\xrightarrow{L} N(0, \pi(1-\pi))
\label{asym_norm2}
\end{equation}
by the central limit theorem. Therefore, by Slutsky's theorem,
\[
\sqrt{n(N-n)}\l \frac{\sum d_i^{U\setminus s}}{\sum d_i^s}\cdot\frac{n(n-1)}{(N-n)n}-\pi\cdot\frac{n(n-1)}{\sum d_i^s} \r\xrightarrow{L} N\l 0, \frac{1-\pi}{\pi} \r.
\]
Multiply both sides by $\sqrt{n(N-n)}/(n-1)$ and Slutsky's theorem yields
\begin{equation*}
n\l \frac{\sum d_i^{U\setminus s}}{\sum d_i^s}-\pi\frac{n(N-n)}{\sum d_i^s} \r\xrightarrow{L} N\l 0, \frac{(1-\pi)(1-c)}{\pi c} \r,
\end{equation*}
which can be rewritten as
\begin{equation}
\left[n\frac{\sum d_i^{U\setminus s}}{\sum d_i^s}-(N-n)\right]+(N-n)\l 1-\frac{\pi}{\sum d_i^s/n^2} \r\xrightarrow{L} N\l 0, \frac{(1-\pi)(1-c)}{\pi c} \r.
\label{decomp3}
\end{equation}

We need to characterize the second term on the left-hand side of (\ref{decomp3}) in order to derive the asymptotic distribution of $\wh N$. By the central limit theorem,
\[
\sqrt{n(n-1)}\l \frac{\sum d_i^s}{n(n-1)}-\pi \r\xrightarrow{L} N(0, 2\pi(1-\pi)),
\]
and therefore
\begin{equation}
(N-n)\l \frac{\sum d_i^s}{n^2}+\frac{\pi}{n}-\pi \r\xrightarrow{L} N\l 0, \frac{2\pi(1-\pi)(1-c)^2}{c^2} \r.
\label{clt3}
\end{equation}
 Define
\[
Y_t=\frac{\sum d_i^s}{n^2}+\frac{\pi}{n}, \quad\text{and}\quad h_t(y)=1-\frac{\pi}{y-\frac{\pi}{n}},
\]
where $n$ is indexed by $t$. Also define $h(y)=1-\pi/y$. Then
\begin{equation}
(N-n)\l 1-\frac{\pi}{\sum d_i^s/n^2} \r=(N-n)[h_t(Y_t)-h(Y_t)]+(N-n)[h(Y_t)-h(\pi)]
\label{decomp4}
\end{equation}
since $h(\pi)=0$. The first term in (\ref{decomp4}) is
\begin{align*}
(N-n)[h_t(Y_t)-h(Y_t)] &= (N-n)\l \frac{\pi}{\sum d_i^s/n^2+\pi/n}-\frac{\pi}{\sum d_i^s/n^2} \r\\
&=-(N-n)\frac{\frac{\pi^2}{n}}{\frac{\sum d_i^s}{n^2}\l\frac{\sum d_i^s}{n^2}+\frac{\pi}{n}\r} \xrightarrow{P} -\frac{1-c}{c},
\end{align*}
and for the second term in (\ref{decomp4}),
\[
(N-n)[h(Y_t)-h(\pi)]\xrightarrow{L} N \l 0, [h'(\pi)]^2\frac{2\pi(1-\pi)(1-c)^2}{c^2} \r
\]
by the delta method. Hence the quantity on the left-hand side of (\ref{decomp4}) satisfies
\begin{equation}
(N-n)\l 1-\frac{\pi}{\sum d_i^s/n^2} \r\xrightarrow{L} N\l -\frac{1-c}{c}, \frac{2(1-\pi)(1-c)^2}{c^2\pi} \r.
\label{equ2}
\end{equation}
Combine (\ref{decomp3}) and (\ref{equ2}), 
\[
n\frac{\sum d_i^{U\setminus s}}{\sum d_i^s}-(N-n)\xrightarrow{L} N\l \frac{1-c}{c}, \tau^2 \r,
\]
where $\tau^2$ is bounded between
\[
\frac{(1-\pi)(1-c)}{\pi c}\left[ 1+\frac{2(1-c)}{c}\pm 2\sqrt\frac{2(1-c)}{c} \right].
\]
$\wh N$ is therefore asymptotically normal with bias $(1-c)/c$ and variance $O(1)$ under outfill. Following from Lemma \ref{norm}, it is inconsistent under the outfill regime.
\end{proof}

\begin{proof}[Proof of Theorem \ref{thm:ht}]
We denote the first and second order inclusion probability of any individual from the $h, l$th cluster as $p_{h}^{(t)}, p_{l}^{(t)}$ and $p_{hl}^{(t)}$ respectively. The superscript $(t)$ corresponds to the sequence of samples and populations specified by the asymptotic regime. Let $X_h$ be the number of individuals sampled from the $h$th cluster. Then $[X_1^{(t)},\ldots, X_H^{(t)}]^T\sim \text{Multinomial} \l n_t, (\frac{1}{H},\ldots,\frac{1}{H})^T \r$ for $t=1, 2, ...$.

The marginal probability that unit $i$ in cluster $h$ is sampled is 
\[ p_{i(h)} = \sum_{j=0}^{n} {n\choose j} \l \frac{1}{H} \r^j \l 1-\frac{1}{H} \r^{n-j}\frac{j}{N_h}=\frac{n}{H\cdot N_h}, \] 
and the joint probability that two units $i, j$ are sampled from clusters $h$ and $l$ ($h\neq l$) is 
\[ p_{i(h)j(l)} = \sum_{p,q} {n\choose {p\ q\ n-p-q}} \l \frac{1}{H} \r^{p}\l \frac{1}{H} \r^{q} \l 1-\frac{2}{H} \r^{n-p-q}\frac{pq}{N_{h}N_{l}}=\frac{n^2-n}{H^2 N_{h} N_{l}}. \]

Since the marginal and joint probabilities are uniform for different $i$ or different combinations of $(i,j)$, we omit the subscripts $i$ or $j$ for simplicity.
We now calculate the variance of the HT estimator $\wh N$ (for one-time sampling). First, according to \citet{horvitz1952generalization}, when $k_t=1$,
\begin{align}
\notag \v\l \wh N^{(t)} \r &=\frac{1}{2}\sum_{h\neq l\in [H_t]} N_h^{(t)}N_l^{(t)} \l p_{h}^{(t)}p_{l}^{(t)}-p_{hl}^{(t)} \r \l \frac{1}{p_{h}^{(t)}}-\frac{1}{p_{l}^{(t)}} \r^2\\
&=\sum_{h=1}^{H_t}\sum_{l=h+1}^{H_t} \frac{\l N_h^{(t)}-N_l^{(t)} \r^2}{n_t}.
\label{var}
\end{align}

\par 1) Under the infill regime, $n_t=n$, $H_t=H$ and $N_h^{(t)}=N_h$ for any $t$, so (\ref{var}) is $O(1)$. The number of samples $k_t$ goes to infinity as $t$ increases, and under $k_t$-time sampling, $\v\l \wh N^{(t)} \r=O(\frac{1}{k_t})$.  
Therefore $\wh N$ is MSE consistent, and also consistent, under infill asymptotics.

\par 2) Under the outfill regime, $n_t=c_tN_t$, $H_t=Ht$ and $N_h^{(t)}=N_h$, where $c_t\rightarrow c\in (0,1)$. The HT estimator is $\wh N^{(t)}=\sum_{h=1}^{H_t} X_h^{(t)}/p_h^{(t)}$, where $\bm X^{(t)}$ is multinomial $\l c_t Nt, (\frac{1}{H_t}, \ldots, \frac{1}{H_t})^T \r$. Then $\wh N^{(t)}\stackrel{d}{=}\sum_{h=1}^H Y_h^{(t)}/p_h^{(t)}$, where $\bm Y^{(t)}$ is multinomial $\l c_tNt, (\frac{1}{H}, \ldots, \frac{1}{H})^T \r$. Then
\begin{equation}
\sqrt{c_tNt}\l \frac{\bm Y^{(t)}}{c_tNt}-\left[ \frac{1}{H}, \ldots, \frac{1}{H} \right]^T \r\xrightarrow{L} MVN(0, \Sigma),
\label{clt4}
\end{equation}
where
\begin{equation*}
\Sigma_{H\times H}=\begin{bmatrix} \frac{1}{H}\l 1-\frac{1}{H} \r & -\frac{1}{H^2} & \cdots &  -\frac{1}{H^2}\\
 -\frac{1}{H^2} &  \frac{1}{H}\l 1-\frac{1}{H}\r & \cdots &  -\frac{1}{H^2}\\
\vdots & \vdots & \ddots & \vdots\\
-\frac{1}{H^2} & -\frac{1}{H^2} & \cdots &  \frac{1}{H}\l 1-\frac{1}{H}\r\end{bmatrix}.
\end{equation*}
Denote $\bm{\omega}^{(t)}=\left[\frac{1}{p_1^{(t)}}, \ldots, \frac{1}{p_{H}^{(t)}}\right]^T=\left[ \frac{HN_1}{c_tN}, \ldots, \frac{HN_H}{c_tN} \right]^T$, then $\wh N^{(t)}={\bm \omega^{(t)}}^T\bm Y^{(t)}$. Also, define $\bm \omega=\lim_{t\rightarrow\infty}\bm\omega^{(t)}=\left[ \frac{HN_1}{cN}, \ldots, \frac{HN_H}{cN} \right]^T$.

Applying the delta method to (\ref{clt4}) yields 
\begin{equation}
\sqrt{c_tNt}\l \frac{\bm\omega^T \bm Y^{(t)}}{c_tNt}- {\bm\omega}^T\left[ \frac{1}{H}, \ldots, \frac{1}{H} \right]^T \r\xrightarrow{L} N(0, \bm\omega^T\Sigma\bm\omega).
\label{delta1}
\end{equation}
Since $c_t\rightarrow c$, by Slutsky's theorem, (\ref{delta1}) leads to

\begin{equation}
\frac{1}{\sqrt t}\l\wh N-Nt\r \xrightarrow{L}N\l 0, cN\sigma^2\r,
\end{equation}
where $\sigma^2={\bm \omega}^T\Sigma \bm\omega=\frac{H\sum_{h=1}^H N_h^2-N^2}{c^2N^2}$. i.e. the variance of $\wh N^{(t)}$ is $O(t)$, which goes to infinity as $t$ increases. It follows from Lemma \ref{norm} that the HT estimator is inconsistent under outfill asymptotics.
\end{proof}

\begin{proof}[Proof of Theorem \ref{chapman}]
Finite-sample claims follow from \citet{chapman1951some}. Setting $n_2=N$ leads to consistency under the finite-population regime. Behavior under infill asymptotics follows from the biases of $\wh N_L$ and $\wh N_C$.

We show the inconsistency of $\wh N_L$ and $\wh N_C$ under the special outfill regime that $n_1=c_1N, n_2=c_2N$ with $N$ increasing, and $n_1, n_2, N$ are indexed by $t$ but the subscripts are omitted for simplicity.

For $\wh N_L$, we prove that $\P(|\wh N_L-N|<\varepsilon)\rightarrow 0$ for some $\varepsilon>0$.
If $c_1c_2N\notin \mathbbm Z$, there exists $b>0$ such that $\lceil c_1c_2N\rceil -c_1c_2N\ge 1/b$ and $c_1c_2N-\lfloor c_1c_2N\rfloor\ge 1/b$. Arbitrarily choose $\eta>0$, pick
\begin{equation*}
0<\varepsilon<\frac{1}{2b(c_1+\eta)(c_2+\eta)}.
\end{equation*}
There exists $T_0>0$ such that 
\begin{equation*}
\frac{n_1n_2}{N(N-\varepsilon)}\le (c_1+\eta)(c_2+\eta), \quad \frac{n_1n_2}{N(N+\varepsilon)}\le (c_1+\eta)(c_2+\eta)
\end{equation*}
for all $t>T_0$. Note that by the choice of $\varepsilon$, 
\begin{equation}
\frac{n_1n_2}{N-\varepsilon}-\frac{n_1n_2}{N}\le \frac{1}{2b}, \quad \frac{n_1n_2}{N}-\frac{n_1n_2}{N+\varepsilon}\le \frac{1}{2b}
\end{equation}
for all $t>T_0$. Then 
\begin{align}
\P(|\wh N_L-N|<\varepsilon) &=\P\left( \frac{n_1n_2}{N+\varepsilon}\le m\le\frac{n_1n_2}{N-\varepsilon} \right)\le \P\left( c_1c_2N-\frac{1}{2b}\le m \le c_1c_2N+\frac{1}{2b}\right).
\label{pf_mbm}
\end{align}
Note that the interval in (\ref{pf_mbm}) contains no integer, i.e. $\P(|\wh N_L-N|<\varepsilon)=0$, if $c_1c_2N$ is not an integer. Otherwise, it contains exactly one integer $c_1c_2N$. Therefore, continuing from (\ref{pf_mbm}), we have (denote $x=c_1c_2N$)
\begin{align}
\notag\P(|\wh N_L-N|<\varepsilon) &\le \P(m=c_1c_2N)=\frac{{n_1\choose x}{N-n_1\choose n_2-x}}{{N\choose n_2}}\\
\notag&=\frac{(N-n_1)!n_1!(N-n_2)!n_2!}{N!(n_1-x)!(n_2-x)!(N-n_1-n_2+x)!x!}\\
&\le \frac{e^4}{2\pi^{5/2}}\frac{(N-n_1)^{N-n_1+\frac{1}{2}}n_1^{n_1+\frac{1}{2}}(N-n_2)^{N-n_2+\frac{1}{2}}n_2^{n_2+\frac{1}{2}}}{N^{N+\frac{1}{2}}(n_1-x)^{n_1-x+\frac{1}{2}}(n_2-x)^{n_2-x+\frac{1}{2}}(N-n_1-n_2+x)^{N-n_1-n_2+x+\frac{1}{2}}x^{x+\frac{1}{2}}}\label{stirling}\\
&:=\frac{e^4}{2\pi^{5/2}}\phi(N)\label{ratio},
\end{align}
where the bound in (\ref{stirling}) is due to Stirling's formula. Consider the function $l(N)=(x(N)+{1}/{2})\log x(N)$. Taking derivative yields $l'(x)=x'(N)\left[\log x(N)+1+\frac{1}{2x(N)}\right]$. Take logarithm in (\ref{ratio}) and we have
\begin{align*}
\frac{\text{d} \log\phi(N)}{\text{d}N}&=(1-c_1)\log(N-n_1)+c_1\log n_1+(1-c_2)\log(N-n_2)+c_2\log n_2\\
&-\log N-c_1(1-c_2)\log(n_1-m)-c_2(1-c_1)\log(n_2-m)\\
&-(1-c_1-c_2+c_1c_2)\log(N-n_1-n_2+m)-c_1c_2\log m\\
&+(1-c_1)+c_1+(1-c_2)+c_2-1-c_1(1-c_2)-c_2(1-c_1)-(1-c_1-c_2+c_1c_2)-c_1c_2-\frac{1}{2N}\\
&=-\frac{1}{2N}<0.
\end{align*}
Also, 
$$\frac{\d^2\log\phi(N)}{\d N^2}=\frac{1}{2N^2}>0,$$ 
so $\log\phi(N)$ is convex on $(0,\infty)$. By the convexity of $\log\phi(N)$ we have
\begin{align*}
\log\phi(N) & \le \log \phi(N-1) + \frac{\d\log\phi(x)}{\d x}\bigg\rvert_{x=N} \\
& \le \log \phi(N-2)+\frac{\d\log\phi(x)}{\d x}\bigg\rvert_{x=N-1}+\frac{\d\log\phi(x)}{\d x}\bigg\rvert_{x=N}\\
& \le \ldots \le \log\phi(1)+\sum_{j=2}^N \frac{\d\log\phi(x)}{\d x}\bigg\rvert_{x=j}\\
& =\log\phi(1)-\sum_{j=2}^N \frac{1}{2j}\rightarrow -\infty,
\end{align*}
which implies that $\P(|\wh N_{L}-N|<\varepsilon)\rightarrow 0$ under the outfill regime when $n_1=c_1N, n_2=c_2N$ for the $\varepsilon>0$ we choose.

Inconsistency of the Chapman CRC estimator under the same regime follows from an essentially identical proof as above. We still pick
$$
0<\varepsilon<\frac{c_1c_2}{4a(c_1+\eta)(c_2+\eta)},
$$
so that there exists $T_0>0$ with
\begin{equation}
\frac{(n_1+1)(n_2+1)}{(N+1)(N-\varepsilon+1)}\le (c_1+\eta)(c_2+\eta), \quad \frac{(n_1+1)(n_2+1)}{(N+1)(N+\varepsilon+1)}\le (c_1+\eta)(c_2+\eta)
\label{pf631}
\end{equation}
for all $t>T_0$, and there exists $T_1>0$ such that
\begin{equation}
c_1c_2N-\frac{c_1c_2}{4a}\le \frac{(n_1+1)(n_2+1)}{N+1}\le c_1c_2N+\frac{c_1c_2}{4a}
\label{pf632}
\end{equation}
for all $t>T_1$. Then, combining (\ref{pf631}) and (\ref{pf632}) yields
\begin{align}
\notag\P(|\wh N_C-N|<\varepsilon) & =\P\left( \frac{(n_1+1)(n_2+1)}{N+\varepsilon+1}-1\le m\le \frac{(n_1+1)(n_2+1)}{N-\varepsilon+1}-1\right)\\
&\le\P\left( c_1c_2\left( N-\frac{1}{2a}\right)-1\le m\le c_1c_2\left( N+\frac{1}{2a} \right)-1 \right)\label{pf633}
\end{align}
for any $t>\max\{T_0, T_1\}$. Then, (\ref{pf633}) is 0 if $c_1c_2N$ is non-integer, and otherwise
\begin{align*}
RHS &=\P\l m=c_1c_2N-1:=x-1 \r\\
& \le \frac{e^4\phi(N)}{2\pi^{5/2}}\cdot\frac{(n_1-x)^{n_1-x+\frac{1}{2}}(n_2-x)^{n_2-x+\frac{1}{2}}(N-n_1-n_2+x)^{N-n_1-n_2+x+\frac{1}{2}}x^{x+\frac{1}{2}}}{(n_1-x+1)^{n1-x+\frac{3}{2}}(n_2-x+1)^{n_2-x+\frac{3}{2}}(N-n_1-n_2+x-1)^{N-n_1-n_2+x-\frac{1}{2}}(x-1)^{x-\frac{1}{2}}}\\
&= \frac{e^4\phi(N)}{2\pi^{5/2}}\cdot\frac{\left( 1-\frac{1}{n_1-x+1}\right)^{n_1-x+\frac{1}{2}}}{n_1-x+1}\cdot\frac{\left( 1-\frac{1}{n_2-x+1} \right)^{n_2-x+\frac{1}{2}}}{n_2-x+1}\cdot\frac{N-n_1-n_2+x}{\left( 1-\frac{1}{N-n_1-n_2+x} \right)^{N-n_1-n_2+x-\frac{1}{2}}}\cdot\frac{x}{\left( 1-\frac{1}{x} \right)^{x-\frac{1}{2}}}\\
&\rightarrow \frac{e^4}{2\pi^{5/2}}\phi(N)\cdot\frac{e^{-1}\cdot e^{-1}\cdot (1-c_1-c_2+c_1c_2)\cdot c_1c_2}{e^{-1}\cdot e^{-1}\cdot(c_1-c_1c_2)(c_2-c_1c_2)}\rightarrow 0,
\end{align*}
 where $\phi(N)$ is defined as in (\ref{ratio}). The argument above implies that $\P(|\wh N_C-N|<\varepsilon)\rightarrow 0$ for the $\varepsilon$ we choose.
\end{proof}

\begin{proof}[Proof of Theorem \ref{asym_mbm}]
The MBM estimator and the Lincoln-Petersen CRC estimator take the same form of $n_1n_2/m$, where $m$ follows hypergeometric distribution with $n_2$ ``draws'', and two categories with sizes $n_1$ and $N-n_1$. Refer to the proof for inconsistency of $\wh N_L$ in Theorem \ref{chapman}.
\end{proof}

\singlespacing
\setlength{\bibsep}{0pt}
\bibliographystyle{custom}
\bibliography{ref}

\begin{thebibliography}{125}
\providecommand{\natexlab}[1]{#1}
\providecommand{\url}[1]{\texttt{#1}}
\expandafter\ifx\csname urlstyle\endcsname\relax
  \providecommand{\doi}[1]{doi: #1}\else
  \providecommand{\doi}{doi: \begingroup \urlstyle{rm}\Url}\fi

\bibitem[Bao et~al.(2015)Bao, Raftery, and Reddy]{bao2015estimating}
Bao, L., Raftery, A.~E., and Reddy, A.
\newblock Estimating the sizes of populations at risk of {HIV} infection from
  multiple data sources using a {B}ayesian hierarchical model.
\newblock \emph{Statistics and Its Interface}, 8\penalty0 (2):\penalty0
  125--136, 2015.

\bibitem[Handcock et~al.(2014)Handcock, Gile, and Mar]{handcock2014estimating}
Handcock, M.~S., Gile, K.~J., and Mar, C.~M.
\newblock Estimating hidden population size using respondent-driven sampling
  data.
\newblock \emph{Electronic Journal of Statistics}, 8\penalty0 (1):\penalty0
  1491, 2014.

\bibitem[UNAIDS and {\relax World Health
  Organization}(2010)]{unaids2010guidelines}
UNAIDS and {\relax World Health Organization}.
\newblock Guidelines on estimating the size of populations most at risk to
  {HIV}.
\newblock Technical report, Geneva, Switzerland, 2010.
\newblock URL
  \url{http://www.unaids.org/en/resources/documents/2011/2011_Estimating_Populations}.

\bibitem[Abdul-Quader et~al.(2014)Abdul-Quader, Baughman, and
  Hladik]{abdul2014estimating}
Abdul-Quader, A.~S., Baughman, A.~L., and Hladik, W.
\newblock Estimating the size of key populations: Current status and future
  possibilities.
\newblock \emph{Current Opinion in HIV and AIDS}, 9\penalty0 (2):\penalty0
  107--114, 2014.

\bibitem[Killworth et~al.(1998{\natexlab{a}})Killworth, McCarty, Bernard,
  Shelley, and Johnsen]{killworth1998estimation}
Killworth, P.~D., McCarty, C., Bernard, H.~R., Shelley, G.~A., and Johnsen,
  E.~C.
\newblock Estimation of seroprevalence, rape, and homelessness in the {U}nited
  {S}tates using a social network approach.
\newblock \emph{Evaluation Review}, 22\penalty0 (2):\penalty0 289--308,
  1998{\natexlab{a}}.

\bibitem[D{\'a}vid and Snijders(2002)]{david2002estimating}
D{\'a}vid, B. and Snijders, T.~A.
\newblock Estimating the size of the homeless population in {Budapest,
  Hungary}.
\newblock \emph{Quality \& Quantity}, 36\penalty0 (3):\penalty0 291--303, 2002.

\bibitem[Shelton(2015)]{shelton2015proposed}
Shelton, J.~F.
\newblock Proposed utilization of the network scale-up method to estimate the
  prevalence of trafficked persons.
\newblock In \emph{Forum on Crime and Society}, volume~8, pages 85--94. United
  Nations Publications, 2015.

\bibitem[van~der Heijden et~al.(2015)van~der Heijden, de~Vries, B{\"o}hning,
  and Cruyff]{van2015estimating}
van~der Heijden, P.~G., de~Vries, I., B{\"o}hning, D., and Cruyff, M.
\newblock Estimating the size of hard-to-reach populations using
  capture-recapture methodology, with a discussion of the {International Labour
  Organization’s} global estimate of forced labour.
\newblock In \emph{Forum on Crime and Society}, volume~8, pages 109--136.
  United Nations Publications, 2015.

\bibitem[Johnston et~al.(2017)Johnston, McLaughlin, Rouhani, and
  Bartels]{johnston2017measuring}
Johnston, L.~G., McLaughlin, K.~R., Rouhani, S.~A., and Bartels, S.~A.
\newblock Measuring a hidden population: A novel technique to estimate the
  population size of women with sexual violence-related pregnancies in {S}outh
  {K}ivu {P}rovince, {D}emocratic {R}epublic of {C}ongo.
\newblock \emph{Journal of Epidemiology and Global Health}, 7\penalty0
  (1):\penalty0 45--53, 2017.

\bibitem[Khalid et~al.(2014)Khalid, Hamad, Othman, Khatib, Mohamed, Ali, and
  Dahoma]{khalid2014estimating}
Khalid, F.~J., Hamad, F.~M., Othman, A.~A., Khatib, A.~M., Mohamed, S., Ali,
  A.~K., and Dahoma, M.~J.
\newblock Estimating the number of people who inject drugs, female sex workers,
  and men who have sex with men, {Unguja Island, Zanzibar}: Results and
  synthesis of multiple methods.
\newblock \emph{AIDS and Behavior}, 18\penalty0 (1):\penalty0 25--31, 2014.

\bibitem[Johnston et~al.(2015)Johnston, McLaughlin, El~Rhilani, Latifi, Toufik,
  Bennani, Alami, Elomari, and Handcock]{johnston2015estimating}
Johnston, L.~G., McLaughlin, K.~R., El~Rhilani, H., Latifi, A., Toufik, A.,
  Bennani, A., Alami, K., Elomari, B., and Handcock, M.~S.
\newblock Estimating the size of hidden populations using respondent-driven
  sampling data: Case examples from {M}orocco.
\newblock \emph{Epidemiology}, 26\penalty0 (6):\penalty0 846, 2015.

\bibitem[Karami et~al.(2017)Karami, Khazaei, Poorolajal, Soltanian, and
  Sajadipoor]{karami2017estimating}
Karami, M., Khazaei, S., Poorolajal, J., Soltanian, A., and Sajadipoor, M.
\newblock Estimating the population size of female sex worker population in
  {T}ehran, {I}ran: Application of direct capture--recapture method.
\newblock \emph{AIDS and Behavior}, 27\penalty0 (8):\penalty0 1--7, 2017.

\bibitem[Vuylsteke et~al.(2017)Vuylsteke, Sika, Semd{\'e}, Anoma, Kacou, and
  Laga]{vuylsteke2017estimating}
Vuylsteke, B., Sika, L., Semd{\'e}, G., Anoma, C., Kacou, E., and Laga, M.
\newblock Estimating the number of female sex workers in {C}{\^o}te d'{I}voire:
  Results and lessons learned.
\newblock \emph{Tropical Medicine and International Health}, 22\penalty0
  (9):\penalty0 1112--1118, 2017.

\bibitem[Ezoe et~al.(2012)Ezoe, Morooka, Noda, Sabin, and
  Koike]{ezoe2012population}
Ezoe, S., Morooka, T., Noda, T., Sabin, M.~L., and Koike, S.
\newblock Population size estimation of men who have sex with men through the
  network scale-up method in {J}apan.
\newblock \emph{PLoS One}, 7\penalty0 (1):\penalty0 e31184, 2012.

\bibitem[Paz-Bailey et~al.(2011)Paz-Bailey, Jacobson, Guardado, Hernandez,
  Nieto, Estrada, and Creswell]{paz2011how}
Paz-Bailey, G., Jacobson, J., Guardado, M., Hernandez, F., Nieto, A., Estrada,
  M., and Creswell, J.
\newblock How many men who have sex with men and female sex workers live in
  {E}l {S}alvador? {U}sing respondent-driven sampling and capture-recapture to
  estimate population sizes.
\newblock \emph{Sexually Transmitted Infections}, 87\penalty0 (4):\penalty0
  279--282, 2011.

\bibitem[Wang et~al.(2015)Wang, Yang, Zhao, Su, Zhao, Chen, Zhang, and
  Zhang]{wang2015application}
Wang, J., Yang, Y., Zhao, W., Su, H., Zhao, Y., Chen, Y., Zhang, T., and Zhang,
  T.
\newblock Application of network scale up method in the estimation of
  population size for men who have sex with men in {S}hanghai, {C}hina.
\newblock \emph{PLoS One}, 10\penalty0 (11):\penalty0 e0143118, 2015.

\bibitem[Wesson et~al.(2015)Wesson, Handcock, McFarland, and
  Raymond]{wesson2015if}
Wesson, P., Handcock, M.~S., McFarland, W., and Raymond, H.~F.
\newblock If you are not counted, you don’t count: Estimating the number of
  {A}frican-{A}merican men who have sex with men in {S}an {F}rancisco using a
  novel {B}ayesian approach.
\newblock \emph{Journal of Urban Health}, 92\penalty0 (6):\penalty0 1052--1064,
  2015.

\bibitem[Quaye et~al.(2015)Quaye, Raymond, Atuahene, Amenyah, Aberle-Grasse,
  McFarland, El-Adas, and {\relax Ghana Men Study Group}]{quaye2015critique}
Quaye, S., Raymond, H.~F., Atuahene, K., Amenyah, R., Aberle-Grasse, J.,
  McFarland, W., El-Adas, A., and {\relax Ghana Men Study Group}.
\newblock Critique and lessons learned from using multiple methods to estimate
  population size of men who have sex with men in {G}hana.
\newblock \emph{AIDS and Behavior}, 19\penalty0 (1):\penalty0 16--23, 2015.

\bibitem[Sabin et~al.(2016)Sabin, Zhao, Calleja, Sheng, Garcia, Reinisch, and
  Komatsu]{sabin2016availability}
Sabin, K., Zhao, J., Calleja, J. M.~G., Sheng, Y., Garcia, S.~A., Reinisch, A.,
  and Komatsu, R.
\newblock Availability and quality of size estimations of female sex workers,
  men who have sex with men, people who inject drugs and transgender women in
  low-and middle-income countries.
\newblock \emph{PLoS One}, 11\penalty0 (5):\penalty0 e0155150, 2016.

\bibitem[Rich et~al.(2018)Rich, Lachowsky, Sereda, Cui, Wong, Wong, Jollimore,
  Raymond, Hottes, Roth, Hogg, and Moore]{rich2017estimating}
Rich, A.~J., Lachowsky, N.~J., Sereda, P., Cui, Z., Wong, J., Wong, S.,
  Jollimore, J., Raymond, H.~F., Hottes, T.~S., Roth, E.~A., Hogg, R.~S., and
  Moore, D.~M.
\newblock Estimating the size of the {MSM} population in {M}etro {V}ancouver,
  {C}anada, using multiple methods and diverse data sources.
\newblock \emph{Journal of Urban Health}, 95\penalty0 (2):\penalty0 188--195,
  2018.

\bibitem[McFarland et~al.(2018)McFarland, Wilson, and
  Raymond]{mcfarland2017many}
McFarland, W., Wilson, E., and Raymond, H.~F.
\newblock How many transgender men are there in {S}an {F}rancisco?
\newblock \emph{Journal of Urban Health}, 95\penalty0 (1):\penalty0 129--133,
  2018.

\bibitem[Kaplan and Soloshatz(1993)]{kaplan1993how}
Kaplan, E.~H. and Soloshatz, D.
\newblock How many drug injectors are there in {N}ew {H}aven? {A}nswers from
  {AIDS} data.
\newblock \emph{Mathematical and Computer Modelling}, 17\penalty0 (2):\penalty0
  109--115, 1993.

\bibitem[Hickman et~al.(2006)Hickman, Hope, Platt, Higgins, Bellis, Rhodes,
  Taylor, and Tilling]{hickman2006estimating}
Hickman, M., Hope, V., Platt, L., Higgins, V., Bellis, M., Rhodes, T., Taylor,
  C., and Tilling, K.
\newblock Estimating prevalence of injecting drug use: A comparison of
  multiplier and capture--recapture methods in cities in {E}ngland and
  {R}ussia.
\newblock \emph{Drug and Alcohol Review}, 25\penalty0 (2):\penalty0 131--140,
  2006.

\bibitem[Kadushin et~al.(2006)Kadushin, Killworth, Bernard, and
  Beveridge]{kadushin2006scale}
Kadushin, C., Killworth, P.~D., Bernard, H.~R., and Beveridge, A.~A.
\newblock Scale-up methods as applied to estimates of heroin use.
\newblock \emph{Journal of Drug Issues}, 36\penalty0 (2):\penalty0 417--440,
  2006.

\bibitem[Heimer and White(2010)]{heimer2010estimation}
Heimer, R. and White, E.
\newblock Estimation of the number of injection drug users in {S}t.
  {P}etersburg, {R}ussia.
\newblock \emph{Drug and Alcohol Dependence}, 109\penalty0 (1):\penalty0
  79--83, 2010.

\bibitem[Salganik et~al.(2011)Salganik, Fazito, Bertoni, Abdo, Mello, and
  Bastos]{salganik2011assessing}
Salganik, M.~J., Fazito, D., Bertoni, N., Abdo, A.~H., Mello, M.~B., and
  Bastos, F.~I.
\newblock Assessing network scale-up estimates for groups most at risk of
  {HIV/AIDS}: Evidence from a multiple-method study of heavy drug users in
  {Curitiba, Brazil}.
\newblock \emph{American Journal of Epidemiology}, 174\penalty0 (10):\penalty0
  1190, 2011.

\bibitem[Nikfarjam et~al.(2016)Nikfarjam, Shokoohi, Shahesmaeili, Haghdoost,
  Baneshi, Haji-Maghsoudi, Rastegari, Nasehi, Memaryan, and
  Tarjoman]{nikfarjam2016national}
Nikfarjam, A., Shokoohi, M., Shahesmaeili, A., Haghdoost, A.~A., Baneshi,
  M.~R., Haji-Maghsoudi, S., Rastegari, A., Nasehi, A.~A., Memaryan, N., and
  Tarjoman, T.
\newblock National population size estimation of illicit drug users through the
  network scale-up method in 2013 in {I}ran.
\newblock \emph{International Journal of Drug Policy}, 31:\penalty0 147--152,
  2016.

\bibitem[Hall et~al.(2000)Hall, Ross, Lynskey, Law, and
  Degenhardt]{hall2000many}
Hall, W.~D., Ross, J.~E., Lynskey, M.~T., Law, M.~G., and Degenhardt, L.~J.
\newblock How many dependent heroin users are there in {A}ustralia?
\newblock \emph{The Medical Journal of Australia}, 173\penalty0 (10):\penalty0
  528--531, 2000.

\bibitem[Yip et~al.(1995)Yip, Bruno, Tajima, Seber, Buckland, Cormack, Unwin,
  Chang, Fienberg, Junker, LaPorte, Libman, and McCarty]{yip1995capture}
Yip, P., Bruno, G., Tajima, N., Seber, G., Buckland, S., Cormack, R., Unwin,
  N., Chang, Y.-F., Fienberg, S., Junker, B., LaPorte, R.~E., Libman, I.~M.,
  and McCarty, D.~J.
\newblock Capture-recapture and multiple-record systems estimation {II}:
  {A}pplications in human diseases.
\newblock \emph{American Journal of Epidemiology}, 142\penalty0 (10):\penalty0
  1059--1068, 1995.

\bibitem[Wittes and Sidel(1968)]{wittes1968generalization}
Wittes, J. and Sidel, V.~W.
\newblock A generalization of the simple capture-recapture model with
  applications to epidemiological research.
\newblock \emph{Journal of Chronic Diseases}, 21\penalty0 (5):\penalty0
  287--301, 1968.

\bibitem[Hook and Regal(1995)]{hook1995capture}
Hook, E.~B. and Regal, R.~R.
\newblock Capture-recapture methods in epidemiology: Methods and limitations.
\newblock \emph{Epidemiologic Reviews}, 17\penalty0 (2):\penalty0 243--264,
  1995.

\bibitem[Robles et~al.(1988)Robles, Marrett, Clarke, and
  Risch]{robles1988application}
Robles, S.~C., Marrett, L.~D., Clarke, E.~A., and Risch, H.~A.
\newblock An application of capture-recapture methods to the estimation of
  completeness of cancer registration.
\newblock \emph{Journal of Clinical Epidemiology}, 41\penalty0 (5):\penalty0
  495--501, 1988.

\bibitem[Karon et~al.(2008)Karon, Song, Brookmeyer, Kaplan, and
  Hall]{karon2008estimating}
Karon, J.~M., Song, R., Brookmeyer, R., Kaplan, E.~H., and Hall, H.~I.
\newblock Estimating {HIV} incidence in the {United States} from {HIV/AIDS}
  surveillance data and biomarker {HIV} test results.
\newblock \emph{Statistics in Medicine}, 27\penalty0 (23):\penalty0 4617--4633,
  2008.

\bibitem[Brookmeyer and Gail(1988)]{brookmeyer1988method}
Brookmeyer, R. and Gail, M.~H.
\newblock A method for obtaining short-term projections and lower bounds on the
  size of the {AIDS} epidemic.
\newblock \emph{Journal of the American Statistical Association}, 83\penalty0
  (402):\penalty0 301--308, 1988.

\bibitem[Seber(1973)]{seber1973estimation}
Seber, G. A.~F.
\newblock \emph{The Estimation of Animal Abundance and Related Parameters}.
\newblock Oxford University Press, 2nd edition, 1973.

\bibitem[Corn and Fogleman(1984)]{Corn1984extinction}
Corn, P.~S. and Fogleman, J.~C.
\newblock Extinction of montane populations of the northern leopard frog
  ({R}ana pipiens) in {C}olorado.
\newblock \emph{Journal of Herpetology}, 18\penalty0 (2):\penalty0 147--152,
  1984.

\bibitem[Hadfield et~al.(1993)Hadfield, Miller, and
  Carwile]{Hadfield1993decimation}
Hadfield, M.~G., Miller, S.~E., and Carwile, A.~H.
\newblock The decimation of endemic {H}awai'ian tree snails by alien predators.
\newblock \emph{American Zoologist}, 33\penalty0 (6):\penalty0 610--622, 1993.

\bibitem[Karanth and Nichols(1998)]{karanth1998estimation}
Karanth, K.~U. and Nichols, J.~D.
\newblock Estimation of tiger densities in {I}ndia using photographic captures
  and recaptures.
\newblock \emph{Ecology}, 79\penalty0 (8):\penalty0 2852--2862, 1998.

\bibitem[Schwarz and Seber(1999)]{schwarz1999review}
Schwarz, C.~J. and Seber, G. A.~F.
\newblock Estimating animal abundance: Review {III}.
\newblock \emph{Statistical Science}, 14\penalty0 (4):\penalty0 427--456, 1999.

\bibitem[Funk et~al.(2003)Funk, Almeida-Reinoso, Nogales-Sornosa, and
  Bustamante]{Funk2003}
Funk, W.~C., Almeida-Reinoso, D., Nogales-Sornosa, F., and Bustamante, M.~R.
\newblock {Monitoring population trends of {E}leutherodactylus frogs}.
\newblock \emph{Journal of Herpetology}, 37\penalty0 (2):\penalty0 245--256,
  2003.

\bibitem[Joglar and Burrowes(1996)]{joglar1996declining}
Joglar, R.~L. and Burrowes, P.~A.
\newblock Declining amphibian populations in {P}uerto {R}ico.
\newblock In Powell, R. and Henderson, R.~W., editors, \emph{Contributions to
  West Indian Herpetology: A tribute to Albert Schwartz}, pages 371--380. The
  Society for the Study of Amphibians and Reptiles, Ithaca, NY, 1996.

\bibitem[Ruggles and Brodie(1947)]{ruggles1947empirical}
Ruggles, R. and Brodie, H.
\newblock An empirical approach to economic intelligence in {W}orld {W}ar {II}.
\newblock \emph{Journal of the American Statistical Association}, 42\penalty0
  (237):\penalty0 72--91, 1947.

\bibitem[Goodman(1952)]{goodman1952serial}
Goodman, L.~A.
\newblock Serial number analysis.
\newblock \emph{Journal of the American Statistical Association}, 47\penalty0
  (260):\penalty0 622--634, 1952.

\bibitem[Davies and Dawson(2014)]{davies2014framework}
Davies, G. and Dawson, S.
\newblock A framework for estimating the number of extremists in {C}anada.
\newblock Technical report, Canadian Network for Research on Terrorism,
  Security, and Society Working Paper Series No. 14-08, 2014.
\newblock URL
  \url{https://www.tsas.ca/working-papers/a-framework-for-estimating-the-number-of-extremists-in-canada/}.

\bibitem[Kaplan(2010)]{kaplan2010terror}
Kaplan, E.~H.
\newblock Terror queues.
\newblock \emph{Operations Research}, 58\penalty0 (4):\penalty0 773--784, 2010.

\bibitem[Kaplan(2012)]{kaplan2012estimating}
Kaplan, E.~H.
\newblock Estimating the duration of {J}ihadi terror plots in the {U}nited
  {S}tates.
\newblock \emph{Studies in Conflict \& Terrorism}, 35\penalty0 (12):\penalty0
  880--894, 2012.

\bibitem[Sadosky et~al.(2015)Sadosky, Shrivastava, Price, and
  Steorts]{sadosky2015blocking}
Sadosky, P., Shrivastava, A., Price, M., and Steorts, R.~C.
\newblock Blocking methods applied to casualty records from the {S}yrian
  {C}onflict.
\newblock \emph{arXiv preprint arXiv:1510.07714}, 2015.

\bibitem[Bernard et~al.(2001)Bernard, Killworth, Johnsen, Shelley, and
  McCarty]{bernard2001estimating}
Bernard, H.~R., Killworth, P.~D., Johnsen, E.~C., Shelley, G.~A., and McCarty,
  C.
\newblock Estimating the ripple effect of a disaster.
\newblock \emph{Connections}, 24\penalty0 (2):\penalty0 18--22, 2001.

\bibitem[Wilson et~al.(2016)Wilson, Sullivan, and Hollis]{wilson2016measuring}
Wilson, J.~M., Sullivan, B.~A., and Hollis, M.~E.
\newblock Measuring the ``unmeasurable'' approaches to assessing the nature and
  extent of product counterfeiting.
\newblock \emph{International Criminal Justice Review}, 26\penalty0
  (3):\penalty0 259--276, 2016.

\bibitem[Horvitz and Thompson(1952)]{horvitz1952generalization}
Horvitz, D.~G. and Thompson, D.~J.
\newblock A generalization of sampling without replacement from a finite
  universe.
\newblock \emph{Journal of the American statistical Association}, 47\penalty0
  (260):\penalty0 663--685, 1952.

\bibitem[Bickel et~al.(1992)Bickel, Nair, and Wang]{bickel1992nonparametric}
Bickel, P.~J., Nair, V.~N., and Wang, P.~C.
\newblock Nonparametric inference under biased sampling from a finite
  population.
\newblock \emph{The Annals of Statistics}, 20\penalty0 (2):\penalty0 853--878,
  1992.

\bibitem[Zheng et~al.(2006)Zheng, Salganik, and Gelman]{zheng2006many}
Zheng, T., Salganik, M.~J., and Gelman, A.
\newblock How many people do you know in prison? {U}sing overdispersion in
  count data to estimate social structure in networks.
\newblock \emph{Journal of the American Statistical Association}, 101\penalty0
  (474):\penalty0 409--423, 2006.

\bibitem[Bernard et~al.(2010)Bernard, Hallett, Iovita, Johnsen, Lyerla,
  McCarty, Mahy, Salganik, Saliuk, Scutelniciuc, Shelley, Sirinirund, Weir, and
  Stroup]{bernard2010counting}
Bernard, H.~R., Hallett, T., Iovita, A., Johnsen, E.~C., Lyerla, R., McCarty,
  C., Mahy, M., Salganik, M.~J., Saliuk, T., Scutelniciuc, O., Shelley, G.~A.,
  Sirinirund, P., Weir, S., and Stroup, D.~F.
\newblock {Counting hard-to-count populations: The network scale-up method for
  public health}.
\newblock \emph{Sexually Transmitted Infections}, 86\penalty0 (Suppl
  2):\penalty0 ii11--15, 2010.

\bibitem[McCormick et~al.(2010)McCormick, Salganik, and
  Zheng]{mccormick2010many}
McCormick, T.~H., Salganik, M.~J., and Zheng, T.
\newblock How many people do you know?: {E}fficiently estimating personal
  network size.
\newblock \emph{Journal of the American Statistical Association}, 105\penalty0
  (489):\penalty0 59--70, 2010.

\bibitem[Feehan and Salganik(2016)]{feehan2016generalizing}
Feehan, D.~M. and Salganik, M.~J.
\newblock Estimating the size of hidden populations using the generalized
  network scale-up estimator.
\newblock \emph{Sociological Methodology}, 46\penalty0 (1):\penalty0 153--186,
  2016.

\bibitem[Chapman(1951)]{chapman1951some}
Chapman, D.~G.
\newblock Some properties of the hypergeometric distribution with applications
  to zoological sample censuses.
\newblock \emph{University of California Publications in Statistics},
  1\penalty0 (7):\penalty0 131--160, 1951.

\bibitem[Darroch(1958)]{darroch1958multiple}
Darroch, J.~N.
\newblock The multiple-recapture census: {I}. {E}stimation of a closed
  population.
\newblock \emph{Biometrika}, 45\penalty0 (3/4):\penalty0 343--359, 1958.

\bibitem[Fienberg(1972)]{fienberg1972multiple}
Fienberg, S.~E.
\newblock The multiple recapture census for closed populations and incomplete
  $2^k$ contingency tables.
\newblock \emph{Biometrika}, 59\penalty0 (3):\penalty0 591--603, 1972.

\bibitem[Pollock et~al.(1990)Pollock, Nichols, Brownie, and
  Hines]{pollock1990statistical}
Pollock, K.~H., Nichols, J.~D., Brownie, C., and Hines, J.~E.
\newblock Statistical inference for capture-recapture experiments.
\newblock \emph{Wildlife Monographs}, 107\penalty0 (1):\penalty0 3--97, 1990.

\bibitem[Zhang et~al.(2007{\natexlab{a}})Zhang, Wang, Lv, Su, Liu, Shen, and
  Bi]{zhang2007advantages}
Zhang, D., Wang, L., Lv, F., Su, W., Liu, Y., Shen, R., and Bi, P.
\newblock Advantages and challenges of using census and multiplier methods to
  estimate the number of female sex workers in a {C}hinese city.
\newblock \emph{AIDS Care}, 19\penalty0 (1):\penalty0 17--19,
  2007{\natexlab{a}}.

\bibitem[Zhang et~al.(2007{\natexlab{b}})Zhang, Lv, Wang, Sun, Zhou, Su, and
  Bi]{zhang2007estimating}
Zhang, D., Lv, F., Wang, L., Sun, L., Zhou, J., Su, W., and Bi, P.
\newblock Estimating the population of female sex workers in two {C}hinese
  cities on the basis of the {HIV/AIDS} behavioural surveillance approach
  combined with a multiplier method.
\newblock \emph{Sexually Transmitted Infections}, 83\penalty0 (3):\penalty0
  228--231, 2007{\natexlab{b}}.

\bibitem[Kimber et~al.(2008)Kimber, Hickman, Degenhardt, Coulson, and
  Van~Beek]{kimber2008estimating}
Kimber, J., Hickman, M., Degenhardt, L., Coulson, T., and Van~Beek, I.
\newblock Estimating the size and dynamics of an injecting drug user population
  and implications for health service coverage: Comparison of indirect
  prevalence estimation methods.
\newblock \emph{Addiction}, 103\penalty0 (10):\penalty0 1604--1613, 2008.

\bibitem[Safarnejad et~al.(2017)Safarnejad, Nga, and
  Son]{safarnejad2017population}
Safarnejad, A., Nga, N.~T., and Son, V.~H.
\newblock Population size estimation of men who have sex with men in {Ho Chi
  Minh City} and {Nghe An} using social app multiplier method.
\newblock \emph{Journal of Urban Health}, 94\penalty0 (3):\penalty0 339--349,
  2017.

\bibitem[Cochran(1977)]{cochran1977sampling}
Cochran, W.~G.
\newblock \emph{Sampling Techniques}.
\newblock Wiley New York, 3rd edition, 1977.

\bibitem[Daniel(1999)]{daniel1999biostatistics}
Daniel, W.~W.
\newblock \emph{Biostatistics: A Foundation for Analysis in the Health
  Sciences}.
\newblock Wiley New York, 7th edition, 1999.

\bibitem[Lwanga and Lemeshow(1991)]{lwanga1991sample}
Lwanga, S.~K. and Lemeshow, S.
\newblock \emph{Sample Size Determination in Health Studies: A Practical
  Manual}.
\newblock Geneva: World Health Organization, 1991.
\newblock URL \url{http://apps.who.int/iris/handle/10665/40062}.

\bibitem[Witte et~al.(1999)Witte, Gauderman, and Thomas]{witte1999asymptotic}
Witte, J.~S., Gauderman, W.~J., and Thomas, D.~C.
\newblock Asymptotic bias and efficiency in case-control studies of candidate
  genes and gene-environment interactions: basic family designs.
\newblock \emph{American Journal of Epidemiology}, 149\penalty0 (8):\penalty0
  693--705, 1999.

\bibitem[Eubank and LaRiccia(1992)]{eubank1992asymptotic}
Eubank, R. and LaRiccia, V.
\newblock Asymptotic comparison of {C}ramer-von {M}ises and nonparametric
  function estimation techniques for testing goodness-of-fit.
\newblock \emph{The Annals of Statistics}, 20\penalty0 (4):\penalty0
  2071--2086, 1992.

\bibitem[Lahiri(1996)]{lahiri1996inconsistency}
Lahiri, S.~N.
\newblock On inconsistency of estimators based on spatial data under infill
  asymptotics.
\newblock \emph{Sankhy{\=a}: The Indian Journal of Statistics, Series A},
  58\penalty0 (3):\penalty0 403--417, 1996.

\bibitem[Mardia and Marshall(1984)]{mardia1984maximum}
Mardia, K.~V. and Marshall, R.~J.
\newblock Maximum likelihood estimation of models for residual covariance in
  spatial regression.
\newblock \emph{Biometrika}, 71\penalty0 (1):\penalty0 135--146, 1984.

\bibitem[Cressie and Lahiri(1993)]{cressie1993asymptotic}
Cressie, N. and Lahiri, S.~N.
\newblock The asymptotic distribution of {REML} estimators.
\newblock \emph{Journal of Multivariate Analysis}, 45\penalty0 (2):\penalty0
  217--233, 1993.

\bibitem[Cressie(2015)]{cressie2015statistics}
Cressie, N.
\newblock \emph{Statistics for Spatial Data}.
\newblock John Wiley \& Sons, 2015.

\bibitem[Stein(2012)]{stein2012interpolation}
Stein, M.~L.
\newblock \emph{Interpolation of Spatial Data: Some Theory for Kriging}.
\newblock Springer Science \& Business Media, 2012.

\bibitem[Zhang(2004)]{zhang2004inconsistent}
Zhang, H.
\newblock Inconsistent estimation and asymptotically equal interpolations in
  model-based geostatistics.
\newblock \emph{Journal of the American Statistical Association}, 99\penalty0
  (465):\penalty0 250--261, 2004.

\bibitem[Chen et~al.(2000)Chen, Simpson, and Ying]{chen2000infill}
Chen, H.-S., Simpson, D.~G., and Ying, Z.
\newblock Infill asymptotics for a stochastic process model with measurement
  error.
\newblock \emph{Statistica Sinica}, 10\penalty0 (1):\penalty0 141--156, 2000.

\bibitem[Isaki and Fuller(1982)]{isaki1982survey}
Isaki, C.~T. and Fuller, W.~A.
\newblock Survey design under the regression superpopulation model.
\newblock \emph{Journal of the American Statistical Association}, 77\penalty0
  (377):\penalty0 89--96, 1982.

\bibitem[Brewer(1979)]{brewer1979class}
Brewer, K. R.~W.
\newblock A class of robust sampling designs for large-scale surveys.
\newblock \emph{Journal of the American Statistical Association}, 74\penalty0
  (368):\penalty0 911--915, 1979.

\bibitem[Gum et~al.(2005)Gum, Lipton, LaPaugh, and Fich]{GUM2005105}
Gum, B., Lipton, R.~J., LaPaugh, A., and Fich, F.
\newblock Estimating the maximum.
\newblock \emph{Journal of Algorithms}, 54\penalty0 (1):\penalty0 105 -- 114,
  2005.

\bibitem[Friedman and Towsley(1999)]{friedman1999multicast}
Friedman, T. and Towsley, D.
\newblock Multicast session membership size estimation.
\newblock In \emph{Proceedings of the 18th Annual Joint Conference of the IEEE
  Computer and Communications Societies}, volume~2 of \emph{INFOCOM'99}, pages
  965--972. IEEE, 1999.

\bibitem[Talluri(2009)]{talluri2009finite}
Talluri, K.
\newblock A finite-population revenue management model and a risk-ratio
  procedure for the joint estimation of population size and parameters.
\newblock Technical report, Universitat Pompeu Fabra, Barcelona, Spain, 2009.
\newblock URL \url{https://ssrn.com/abstract=1374853}.

\bibitem[Blumenthal and Dahiya(1981)]{Blumenthal1980}
Blumenthal, S. and Dahiya, R.~C.
\newblock Estimating the binomial parameter $n$.
\newblock \emph{Journal of the American Statistical Association}, 76\penalty0
  (376):\penalty0 903--909, 1981.

\bibitem[Feldman and Fox(1968)]{Feldman1968}
Feldman, D. and Fox, M.
\newblock Estimation of the parameter $n$ in the binomial distribution.
\newblock \emph{Journal of American Statistical Association}, 63\penalty0
  (321):\penalty0 150-- 158, 1968.

\bibitem[Kahn(1987)]{kahn1987cautionary}
Kahn, W.~D.
\newblock A cautionary note for {B}ayesian estimation of the binomial parameter
  $n$.
\newblock \emph{The American Statistician}, 41\penalty0 (1):\penalty0 38--40,
  1987.

\bibitem[Cruyff and van~der Heijden(2008)]{cruyff2008point}
Cruyff, M.~J. and van~der Heijden, P.~G.
\newblock Point and interval estimation of the population size using a
  zero-truncated negative binomial regression model.
\newblock \emph{Biometrical Journal}, 50\penalty0 (6):\penalty0 1035--1050,
  2008.

\bibitem[B{\"o}hning et~al.(2004)B{\"o}hning, Suppawattanabodee, Kusolvisitkul,
  and Viwatwongkasem]{bohning2004estimating}
B{\"o}hning, D., Suppawattanabodee, B., Kusolvisitkul, W., and Viwatwongkasem,
  C.
\newblock Estimating the number of drug users in {B}angkok 2001: A
  capture-recapture approach using repeated entries in one list.
\newblock \emph{European Journal of Epidemiology}, 19\penalty0 (12):\penalty0
  1075, 2004.

\bibitem[van~der Heijden et~al.(2003{\natexlab{a}})van~der Heijden, Bustami,
  Cruyff, Engbersen, and van Houwelingen]{van2003point}
van~der Heijden, P.~G., Bustami, R., Cruyff, M.~J., Engbersen, G., and van
  Houwelingen, H.~C.
\newblock Point and interval estimation of the population size using the
  truncated {P}oisson regression model.
\newblock \emph{Statistical Modelling}, 3\penalty0 (4):\penalty0 305--322,
  2003{\natexlab{a}}.

\bibitem[B{\"o}hning and van~der Heijden(2009)]{bohning2009covariate}
B{\"o}hning, D. and van~der Heijden, P.~G.
\newblock A covariate adjustment for zero-truncated approaches to estimating
  the size of hidden and elusive populations.
\newblock \emph{The Annals of Applied Statistics}, 3\penalty0 (2):\penalty0
  595--610, 2009.

\bibitem[van~der Heijden et~al.(2003{\natexlab{b}})van~der Heijden, Cruyff, and
  van Houwelingen]{van2003estimating}
van~der Heijden, P.~G., Cruyff, M., and van Houwelingen, H.~C.
\newblock Estimating the size of a criminal population from police records
  using the truncated {P}oisson regression model.
\newblock \emph{Statistica Neerlandica}, 57\penalty0 (3):\penalty0 289--304,
  2003{\natexlab{b}}.

\bibitem[Bouchard(2007)]{bouchard2007capture}
Bouchard, M.
\newblock A capture-recapture model to estimate the size of criminal
  populations and the risks of detection in a marijuana cultivation industry.
\newblock \emph{Journal of Quantitative Criminology}, 23\penalty0 (3):\penalty0
  221--241, 2007.

\bibitem[Scollnik(1997)]{scollnik1997inference}
Scollnik, D.~P.
\newblock Inference concerning the size of the zero class from an incomplete
  {P}oisson sample.
\newblock \emph{Communications in Statistics -- Theory and Methods},
  26\penalty0 (1):\penalty0 221--236, 1997.

\bibitem[Wilson and Collins(1992)]{wilson1992capture}
Wilson, R.~M. and Collins, M.~F.
\newblock Capture-recapture estimation with samples of size one using frequency
  data.
\newblock \emph{Biometrika}, 79\penalty0 (3):\penalty0 543--553, 1992.

\bibitem[Craig(1953)]{craig1953utilization}
Craig, C.~C.
\newblock On the utilization of marked specimens in estimating populations of
  flying insects.
\newblock \emph{Biometrika}, 40\penalty0 (1/2):\penalty0 170--176, 1953.

\bibitem[Frey and Kaplan(2010)]{frey2010queue}
Frey, J.~C. and Kaplan, E.~H.
\newblock Queue inference from periodic reporting data.
\newblock \emph{Operations Research Letters}, 38\penalty0 (5):\penalty0
  420--426, 2010.

\bibitem[Crawford(2016)]{crawford2016graphical}
Crawford, F.~W.
\newblock The graphical structure of respondent-driven sampling.
\newblock \emph{Sociological Methodology}, 46\penalty0 (1):\penalty0 187--211,
  2016.

\bibitem[Crawford et~al.(2018)Crawford, Wu, and Heimer]{crawford2017hidden}
Crawford, F.~W., Wu, J., and Heimer, R.
\newblock Hidden population size estimation from respondent-driven sampling: A
  network approach.
\newblock \emph{Journal of the American Statistical Association}, 2018.
\newblock In press.

\bibitem[Killworth et~al.(1998{\natexlab{b}})Killworth, Johnsen, McCarty,
  Shelley, and Bernard]{killworth1998social}
Killworth, P.~D., Johnsen, E.~C., McCarty, C., Shelley, G.~A., and Bernard,
  H.~R.
\newblock A social network approach to estimating seroprevalence in the {United
  States}.
\newblock \emph{Social Networks}, 20\penalty0 (1):\penalty0 23--50,
  1998{\natexlab{b}}.

\bibitem[Katzir et~al.(2011)Katzir, Liberty, and Somekh]{katzir2011estimating}
Katzir, L., Liberty, E., and Somekh, O.
\newblock Estimating sizes of social networks via biased sampling.
\newblock In \emph{Proceedings of the 20th International Conference on World
  Wide Web}, pages 597--606. ACM, 2011.

\bibitem[Bernard et~al.(1991)Bernard, Johnsen, Killworth, and
  Robinson]{bernard1991estimating}
Bernard, H.~R., Johnsen, E.~C., Killworth, P.~D., and Robinson, S.
\newblock Estimating the size of an average personal network and of an event
  subpopulation: Some empirical results.
\newblock \emph{Social Science Research}, 20\penalty0 (2):\penalty0 109--121,
  1991.

\bibitem[Massouli{\'e} et~al.(2006)Massouli{\'e}, Le~Merrer, Kermarrec, and
  Ganesh]{Massoulie2006peer}
Massouli{\'e}, L., Le~Merrer, E., Kermarrec, A.-M., and Ganesh, A.
\newblock Peer counting and sampling in overlay networks: Random walk methods.
\newblock In \emph{Proceedings of the 25th Annual ACM Symposium on Principles
  of Distributed Computing}, PODC '06, pages 123--132. ACM, 2006.

\bibitem[Erd\H{o}s and R\'{e}nyi(1959)]{erdos1959random}
Erd\H{o}s, P. and R\'{e}nyi, A.
\newblock On random graphs {I}.
\newblock \emph{Publicationes Mathematicae}, 6:\penalty0 290--297, 1959.

\bibitem[Chen et~al.(2016)Chen, Karbasi, and Crawford]{chen2016estimating}
Chen, L., Karbasi, A., and Crawford, F.~W.
\newblock Estimating the size of a large network and its communities from a
  random sample.
\newblock In \emph{Advances in Neural Information Processing Systems 29}, pages
  3072--3080. Curran Associates, Inc., 2016.

\bibitem[Borchers et~al.(1998)Borchers, Buckland, Goedhart, Clarke, and
  Hedley]{borchers1998horvitz}
Borchers, D.~L., Buckland, S.~T., Goedhart, P.~W., Clarke, E.~D., and Hedley,
  S.~L.
\newblock {H}orvitz-{T}hompson estimators for double-platform line transect
  surveys.
\newblock \emph{Biometrics}, 54\penalty0 (4):\penalty0 1221--1237, 1998.

\bibitem[Pollock(1982)]{pollock1982capture}
Pollock, K.~H.
\newblock A capture-recapture design robust to unequal probability of capture.
\newblock \emph{The Journal of Wildlife Management}, 46\penalty0 (3):\penalty0
  752--757, 1982.

\bibitem[Chao(1987)]{chao1987cap}
Chao, A.
\newblock Estimating the population size for capture-recapture data with
  unequal catchability.
\newblock \emph{Biometrics}, 43\penalty0 (4):\penalty0 783--791, 1987.

\bibitem[Jolly(1965)]{jolly1965explicit}
Jolly, G.~M.
\newblock Explicit estimates from capture-recapture data with both death and
  immigration-stochastic model.
\newblock \emph{Biometrika}, 52\penalty0 (1/2):\penalty0 225--247, 1965.

\bibitem[Schwarz and Arnason(1996)]{schwarz1996general}
Schwarz, C.~J. and Arnason, A.~N.
\newblock A general methodology for the analysis of capture-recapture
  experiments in open populations.
\newblock \emph{Biometrics}, 52\penalty0 (3):\penalty0 860--873, 1996.

\bibitem[Young and Young(1998)]{young1998capture}
Young, L.~J. and Young, J.~H.
\newblock Capture-recapture: Open populations.
\newblock In \emph{Statistical Ecology}, pages 357--389. Springer, 1998.

\bibitem[Khan et~al.(2018)Khan, Lee, and Dombrowski]{khan2017one}
Khan, B., Lee, H.-W., and Dombrowski, K.
\newblock One-step estimation of networked population size with anonymity using
  respondent-driven capture-recapture and hashing.
\newblock \emph{PLoS One}, 13\penalty0 (4):\penalty0 e0195959, 2018.

\bibitem[Lincoln(1930)]{lincoln1930calculating}
Lincoln, F.~C.
\newblock \emph{Calculating Waterfowl Abundance on the Basis of Banding
  Returns}.
\newblock U.S. Department of Agriculture, Washington, D.C., 1930.

\bibitem[Petersen(1894)]{petersen1894biology}
Petersen, C. G.~J.
\newblock \emph{On the Biology of Our Flatfishes and on the Decrease of Our
  Flat-Fish Fisheries: With Some Observations Showing How to Remedy the Latter
  and Promote the Flat-Fish Fisheries in Our Seas East of the Skaw}.
\newblock Centraltrykkeriet, 1894.

\bibitem[Godfrey et~al.(2002)Godfrey, Eaton, McDougall, and
  Culyer]{godfrey2002economic}
Godfrey, C., Eaton, G., McDougall, C., and Culyer, A.
\newblock \emph{The Economic and Social Costs of Class {A} Drug Use in
  {E}ngland and {W}ales, 2000}.
\newblock Home Office London, 2002.

\bibitem[Frischer et~al.(2001)Frischer, Hickman, Kraus, Mariani, and
  Wiessing]{frischer2001comparison}
Frischer, M., Hickman, M., Kraus, L., Mariani, F., and Wiessing, L.
\newblock A comparison of different methods for estimating the prevalence of
  problematic drug misuse in {Great Britain}.
\newblock \emph{Addiction}, 96\penalty0 (10):\penalty0 1465--1476, 2001.

\bibitem[Robson and Regier(1964)]{robson1964sample}
Robson, D. and Regier, H.
\newblock Sample size in {P}etersen mark--recapture experiments.
\newblock \emph{Transactions of the American Fisheries Society}, 93\penalty0
  (3):\penalty0 215--226, 1964.

\bibitem[Jensen(1981)]{jensen1981sample}
Jensen, A.
\newblock Sample sizes for single mark and single recapture experiments.
\newblock \emph{Transactions of the American Fisheries Society}, 110\penalty0
  (3):\penalty0 455--458, 1981.

\bibitem[Bailey(1951)]{bailey1951estimating}
Bailey, N. T.~J.
\newblock On estimating the size of mobile populations from recapture data.
\newblock \emph{Biometrika}, 38\penalty0 (3/4):\penalty0 293--306, 1951.

\bibitem[Brownie and Pollock(1985)]{brownie1985analysis}
Brownie, C. and Pollock, K.~H.
\newblock Analysis of multiple capture-recapture data using band-recovery
  methods.
\newblock \emph{Biometrics}, 41\penalty0 (2):\penalty0 411--420, 1985.

\bibitem[Mills et~al.(2000)Mills, Citta, Lair, Schwartz, and
  Tallmon]{mills2000estimating}
Mills, L.~S., Citta, J.~J., Lair, K.~P., Schwartz, M.~K., and Tallmon, D.~A.
\newblock Estimating animal abundance using noninvasive {DNA} sampling: promise
  and pitfalls.
\newblock \emph{Ecological Applications}, 10\penalty0 (1):\penalty0 283--294,
  2000.

\bibitem[Vincent and Thompson(2014)]{vincent2014estimating}
Vincent, K. and Thompson, S.
\newblock Estimating the size and distribution of networked populations with
  snowball sampling.
\newblock \emph{arXiv preprint arXiv:1402.4372}, 2014.

\bibitem[Vincent and Thompson(2017)]{vincent2016estimating}
Vincent, K. and Thompson, S.
\newblock Estimating population size with link-tracing sampling.
\newblock \emph{Journal of the American Statistical Association}, 112\penalty0
  (519):\penalty0 1286--1295, 2017.

\bibitem[Bunge et~al.(2014)Bunge, Willis, and Walsh]{bunge2014estimating}
Bunge, J., Willis, A., and Walsh, F.
\newblock Estimating the number of species in microbial diversity studies.
\newblock \emph{Annual Review of Statistics and Its Application}, 1\penalty0
  (1):\penalty0 427--445, 2014.

\bibitem[Fusy and Giroire(2007)]{fusy2007estimating}
Fusy, E. and Giroire, F.
\newblock Estimating the number of active flows in a data stream over a sliding
  window.
\newblock In \emph{Proceedings of the Meeting on Analytic Algorithmics and
  Combinatorics}, ANALCO '07, pages 223--231, Philadelphia, USA, Jan 2007.
  Society for Industrial and Applied Mathematics.

\bibitem[Chassaing and Gerin(2006)]{chassaing2006}
Chassaing, P. and Gerin, L.
\newblock Efficient estimation of the cardinality of large data sets.
\newblock In \emph{Proceedings of the Fourth Colloquium on Mathematics and
  Computer Science}, pages 419--422. Discrete Mathematics and Theoretical
  Computer Science, Aug 2006.

\bibitem[Kane et~al.(2010)Kane, Nelson, and Woodruff]{Kane2010}
Kane, D.~M., Nelson, J., and Woodruff, D.~P.
\newblock An optimal algorithm for the distinct elements problem.
\newblock In \emph{Proceedings of the Twenty-ninth ACM SIGMOD-SIGACT-SIGART
  Symposium on Principles of Database Systems}, PODS '10, pages 41--52, New
  York, USA, Jun 2010. Association for Computing Machinery.

\bibitem[Creel et~al.(2003)Creel, Spong, Sands, Rotella, Zeigle, Joe, Murphy,
  and Smith]{creel2003population}
Creel, S., Spong, G., Sands, J.~L., Rotella, J., Zeigle, J., Joe, L., Murphy,
  K.~M., and Smith, D.
\newblock Population size estimation in {Y}ellowstone wolves with error-prone
  noninvasive microsatellite genotypes.
\newblock \emph{Molecular Ecology}, 12\penalty0 (7):\penalty0 2003--2009, 2003.

\bibitem[Bellemain et~al.(2005)Bellemain, Swenson, Tallmon, Brunberg, and
  Taberlet]{bellemain2005estimating}
Bellemain, E., Swenson, J.~E., Tallmon, D., Brunberg, S., and Taberlet, P.
\newblock Estimating population size of elusive animals with {DNA} from
  hunter-collected feces: Four methods for brown bears.
\newblock \emph{Conservation Biology}, 19\penalty0 (1):\penalty0 150--161,
  2005.

\end{thebibliography}

\end{document}